\documentclass[a4paper,twoside,12pt]{article}
\usepackage[english]{babel}

\usepackage{amsmath,amsfonts,amssymb,amsthm}
\newtheorem{teo}{Theorem}

\newtheorem{remark}{Remark}

\newtheorem{ex}{Example}

\newtheorem{corol}{Corollary}

 \title{{\bf    Schatten-von Neumann classes with more subtle asymptotics  than one of the power type   }}

\author{Maksim \,V.~Kukushkin   \\ \\
 \small  \textit{Moscow State University of Civil Engineering, 129337,  Moscow, Russia}\\
 \textit{\small\textit{kukushkinmv@rambler.ru}} }

\date{}
\begin{document}

\maketitle

\begin{abstract}In this paper we study non-selfadjoint operators using the methods of the spectral theory. The main challenge is to represent a complete description  of an operator  belonging to the Schatten-von Neumann class  having used  the  order of the Hermitian real component. The latter fundamental result is advantageous since many theoretical statements based upon it and  one of them is a concept on the root series expansion to say nothing on a wide spectrum of applications in the theory of evolution equations. In this regard the evolution equations of fractional order with the operator function in the term not containing the time variable are involved. Constructing the operator class satisfying the asymptotic more subtle than one  of the power type, we show convexly the relevance of the obtained fundamental results.

\end{abstract}
\begin{small}\textbf{Keywords:}
 Strictly accretive operator;  Abel-Lidskii basis property;   Schatten-von Neumann  class; convergence exponent; counting function.   
\end{small}

\section{Introduction}

  Generally, the concept origins from the  well-known fact that the   eigenvectors   of the compact selfadjoint operator form a basis in the closure of its range. The question what happens in the case when the operator is non-selfadjoint is rather complicated and  deserves to be considered as a separate part of the spectral theory. Basically, the aim of the mentioned  part of the spectral theory   are   propositions on the convergence of the root vector series in one or another sense to an element belonging to the closure of the operator range. Here, we should note when we  say a sense we mean   Bari, Riesz, Abel (Abel-Lidskii) senses of the series convergence  \cite{firstab_lit:2Agranovich1994},\cite{firstab_lit:1Gohberg1965}. A reasonable question that appears is about minimal conditions that guaranty the desired result, for instance in the mentioned papers    the authors  considered a domain of the parabolic type containing the spectrum of the operator. In the paper \cite{firstab_lit:2Agranovich1994}, non-salfadjoint operators with the special condition imposed on the numerical range of values are considered. The main advantage of this result is a   weak condition imposed upon the numerical range of values comparatively with the sectorial condition (see definition of the sectorial operator). Thus, the convergence in the Abel-Lidskii sense  was established for an operator class wider than the class of sectorial operators. Here, we  make a comparison between results devoted to operators with the discrete spectra and operators with the compact resolvent, for they can be easily reformulated from one to another realm.

  The central idea of this paper  is to formulate sufficient conditions of the Abel-Lidskii basis property of the root functions system for a sectorial non-selfadjoint operator of the special type. Considering such an operator class, we strengthen a little the condition regarding the semi-angle of the sector, but weaken a great deal conditions regarding the involved parameters. Moreover, the central aim generates  some prerequisites to consider technical peculiarities such as a newly constructed sequence of contours of the power type on the contrary to the Lidskii results \cite{firstab_lit:1Lidskii}, where a sequence of the contours of the  exponential type was considered. Thus,  we clarify   results  \cite{firstab_lit:1Lidskii} devoted to the decomposition on the root vector system of the non-selfadjoint operator. We use  a technique of the entire function theory and introduce  a so-called  Schatten-von Neumann class of the convergence  exponent. Considering strictly accretive operators satisfying special conditions formulated in terms of the norm,   using a  sequence of contours of the power type,  we invent a peculiar  method how to calculate  a contour integral involved in the problem in its general statement.     Finally, we produce applications to differential equations in the abstract Hilbert space.
  In particular,     the existence and uniqueness theorems  for evolution equations   with the right-hand side -- a  differential operator  with a fractional derivative in  final terms are covered by the invented abstract method. In this regard such operator  as a Riemann-Liouville  fractional differential operator,    Kipriyanov operator, Riesz potential,  difference operator are involved.
Note that analysis of the required  conditions imposed upon the right-hand side of the    evolution equations that are in the scope  leads us to   relevance of the central idea of the paper. Here, we should note  a well-known fact  \cite{firstab_lit:Shkalikov A.},\cite{firstab_lit(arXiv non-self)kukushkin2018}   that a particular interest appears in the case when a senior term of the operator at least
    is not selfadjoint,  for in the contrary case there is a plenty of results devoted to the topic wherein  the following papers are well-known
\cite{firstab_lit:1Katsnelson},\cite{firstab_lit:1Krein},\cite{firstab_lit:Markus Matsaev},\cite{firstab_lit:2Markus},\cite{firstab_lit:Shkalikov A.}. The fact is that most of them deal with a decomposition of the  operator  on a sum,  where the senior term
     must be either a selfadjoint or normal operator. In other cases, the  methods of the papers
     \cite{kukushkin2019}, \cite{firstab_lit(arXiv non-self)kukushkin2018} become relevant   and allow us  to study spectral properties of   operators  whether we have the mentioned above  representation or not.   We should remark that the results of  the papers \cite{firstab_lit:2Agranovich1994},\cite{firstab_lit:Markus Matsaev}, applicable to study non-selfadjoin operators,  are   based on the sufficiently strong assumption regarding the numerical range of values of the operator. At the same time,
the methods  \cite{firstab_lit(arXiv non-self)kukushkin2018}   can be  used    in  the  natural way,  if we deal with more abstract constructions formulated in terms of the semigroup theory   \cite{kukushkin2021a}.  The  central challenge  of the latter  paper  is how  to create a model   representing     a  composition of  fractional differential operators   in terms of the semigroup theory.   We should note that motivation arouse in connection with the fact that
  a second order differential operator can be represented  as a some kind of  a  transform of   the infinitesimal generator of a shift semigroup. Here, we should stress   that
  the eigenvalue problem for the operator
     was previously  studied by methods of  theory of functions   \cite{firstab_lit:1Nakhushev1977}.
Thus, having been inspired by   novelty of the  idea  we generalize a   differential operator with a fractional integro-differential composition  in the final terms   to some transform of the corresponding  infinitesimal generator of the shift semigroup.
By virtue of the   methods obtained in the paper
\cite{firstab_lit(arXiv non-self)kukushkin2018}, we   managed  to  study spectral properties of the  infinitesimal generator  transform and obtain   an outstanding result --
   asymptotic equivalence between   the
real component of the resolvent and the resolvent of the   real component of the operator. The relevance is based on the fact that
   the  asymptotic formula  for   the operator  real component  can be  established in most  cases due to well-known asymptotic relations  for the regular differential operators as well as for the singular ones
 \cite{firstab_lit:Rosenblum}. It is remarkable that the results establishing spectral properties of non-selfadjoint operators  allow  us to implement a novel approach regarding   the problem  of the basis property of  root vectors.
   In its own turn, the application  of results connected with the basis property  covers  many   problems in the framework of the theory of evolution  equations.
  The abstract approach to the Cauchy problem for the fractional evolution equation was previously implemented in the papers \cite{firstab_lit:Bazhl},\cite{Ph. Cl}. At the same time, the  main advantage of this paper is the obtained formula for the solution of the evolution equation with the relatively wide conditions imposed upon the right-hand side,  where  the derivative at the left-hand side is supposed to be of the fractional order.
 This problem appeals to many ones that lie  in the framework of the theory of differential equations, for instance in the paper   \cite{L. Mor} the solution of the  evolution equation  can be obtained in the analytical way if we impose the conditions upon the right-hand side. We can also produce a number  of papers dealing  with differential equations which can be studied by this paper abstract methods \cite{firstab_lit:Mainardi F.}, \cite{firstab_lit:Mamchuev2017a}, \cite{firstab_lit:Mamchuev2017}, \cite{firstab_lit:Pskhu},       \cite{firstab_lit:Wyss}. The latter information gives us an opportunity to claim that the offered  approach is undoubtedly novel and relevant.

\section{Preliminaries}

Let    $ C,C_{i} ,\;i\in \mathbb{N}_{0}$ be   real constants. We   assume   that  a  value of $C$ is positive and   can be different in   various formulas  but   values of $C_{i} $ are  certain. Denote by $ \mathrm{int} \,M,\;\mathrm{Fr}\,M$ the interior and the set of boundary points of the set $M$ respectively.   Everywhere further, if the contrary is not stated, we consider   linear    densely defined operators acting on a separable complex  Hilbert space $\mathfrak{H}$. Denote by $ \mathcal{B} (\mathfrak{H})$    the set of linear bounded operators   on    $\mathfrak{H}.$  Denote by
    $\tilde{L}$   the  closure of an  operator $L.$ We establish the following agreement on using  symbols $\tilde{L}^{i}:= (\tilde{L})^{i},$ where $i$ is an arbitrary symbol.  Denote by    $    \mathrm{D}   (L),\,   \mathrm{R}   (L),\,\mathrm{N}(L)$      the  {\it domain of definition}, the {\it range},  and the {\it kernel} or {\it null space}  of an  operator $L$ respectively. The deficiency (codimension) of $\mathrm{R}(L),$ dimension of $\mathrm{N}(L)$   are denoted by $\mathrm{def}\, L,\;\mathrm{nul}\,L$ respectively.  Assume that $L$ is a closed   operator acting on $\mathfrak{H},\,\mathrm{N}(L)=0,$  let us define a Hilbert space
$
 \mathfrak{H}_{L}:= \big \{f,g\in \mathrm{D}(L),\,(f,g)_{ \mathfrak{H}_{L}}=(Lf,Lg)_{\mathfrak{H} } \big\}.
$
Consider a pair of complex Hilbert spaces $\mathfrak{H},\mathfrak{H}_{+},$ the notation
$
\mathfrak{H}_{+}\subset\subset\mathfrak{ H}
$
   means that $\mathfrak{H}_{+}$ is dense in $\mathfrak{H}$ as a set of    elements and we have a bounded embedding provided by the inequality
$$
\|f\|_{\mathfrak{H}}\leq C_{0}\|f\|_{\mathfrak{H}_{+}},\,C_{0}>0,\;f\in \mathfrak{H}_{+},
$$
moreover   any  bounded  set with respect to the norm $\mathfrak{H}_{+}$ is compact with respect to the norm $\mathfrak{H}.$
  Let $L$ be a closed operator, for any closable operator $S$ such that
$\tilde{S} = L,$ its domain $\mathrm{D} (S)$ will be called a core of $L.$ Denote by $\mathrm{D}_{0}(L)$ a core of a closeable operator $L.$ Let    $\mathrm{P}(L)$ be  the resolvent set of an operator $L$ and
     $ R_{L}(\zeta),\,\zeta\in \mathrm{P}(L),\,[R_{L} :=R_{L}(0)]$ denotes      the resolvent of an  operator $L.$ Denote by $\lambda_{i}(L),\,i\in \mathbb{N} $ the eigenvalues of an operator $L.$
 Suppose $L$ is  a compact operator and  $N:=(L^{\ast}L)^{1/2},\,r(N):={\rm dim}\,  \mathrm{R}  (N);$ then   the eigenvalues of the operator $N$ are called   the {\it singular  numbers} ({\it s-numbers}) of the operator $L$ and are denoted by $s_{i}(L),\,i=1,\,2,...\,,r(N).$ If $r(N)<\infty,$ then we put by definition     $s_{i}=0,\,i=r(N)+1,2,...\,.$
 According  to the terminology of the monograph   \cite{firstab_lit:1Gohberg1965}  the  dimension  of the  root vectors subspace  corresponding  to a certain  eigenvalue $\lambda_{k}$  is called  the {\it algebraic multiplicity} of the eigenvalue $\lambda_{k}.$
Let  $\nu(L)$ denotes   the sum of all algebraic multiplicities of an  operator $L.$ Denote by $n(r)$ a function equals to a number of the elements of the sequence $\{a_{n}\}_{1}^{\infty},\,|a_{n}|\uparrow\infty$ within the circle $|z|<r.$ Let $A$ be a compact operator, denote by $n_{A}(r)$   {\it counting function}   a function $n(r)$ corresponding to the sequence  $\{s^{-1}_{i}(A)\}_{1}^{\infty}.$ Let  $\mathfrak{S}_{p}(\mathfrak{H}),\, 0< p<\infty $ be       a Schatten-von Neumann    class and      $\mathfrak{S}_{\infty}(\mathfrak{H})$ be the set of compact operators.
    Suppose  $L$ is  an   operator with a compact resolvent and
$s_{n}(R_{L})\leq   C \,n^{-\mu},\,n\in \mathbb{N},\,0\leq\mu< \infty;$ then
 we
 denote by  $\mu(L) $   order of the     operator $L$ in accordance with  the definition given in the paper  \cite{firstab_lit:Shkalikov A.}.
 Denote by  $ \mathfrak{Re} L  := \left(L+L^{*}\right)/2,\, \mathfrak{Im} L  := \left(L-L^{*}\right)/2 i$
  the  real  and   imaginary components    of an  operator $L$  respectively.
In accordance with  the terminology of the monograph  \cite{firstab_lit:kato1980} the set $\Theta(L):=\{z\in \mathbb{C}: z=(Lf,f)_{\mathfrak{H}},\,f\in  \mathrm{D} (L),\,\|f\|_{\mathfrak{H}}=1\}$ is called the  {\it numerical range}  of an   operator $L.$
  An  operator $L$ is called    {\it sectorial}    if its  numerical range   belongs to a  closed
sector     $\mathfrak{ L}_{\iota}(\theta):=\{\zeta:\,|\arg(\zeta-\iota)|\leq\theta<\pi/2\} ,$ where      $\iota$ is the vertex   and  $ \theta$ is the semi-angle of the sector   $\mathfrak{ L}_{\iota}(\theta).$ If we want to stress the  correspondence  between $\iota$ and $\theta,$  then   we will write $\theta_{\iota}.$
 An operator $L$ is called  {\it bounded from below}   if the following relation  holds  $\mathrm{Re}(Lf,f)_{\mathfrak{H}}\geq \gamma_{L}\|f\|^{2}_{\mathfrak{H}},\,f\in  \mathrm{D} (L),\,\gamma_{L}\in \mathbb{R},$  where $\gamma_{L}$ is called a lower bound of $L.$ An operator $L$ is called  {\it   accretive}   if  $\gamma_{L}=0.$
 An operator $L$ is called  {\it strictly  accretive}   if  $\gamma_{L}>0.$      An  operator $L$ is called    {\it m-accretive}     if the next relation  holds $(A+\zeta)^{-1}\in \mathcal{B}(\mathfrak{H}),\,\|(A+\zeta)^{-1}\| \leq   (\mathrm{Re}\zeta)^{-1},\,\mathrm{Re}\zeta>0. $
An operator $L$ is called    {\it m-sectorial}   if $L$ is   sectorial    and $L+ \beta$ is m-accretive   for some constant $\beta.$   An operator $L$ is called     {\it symmetric}     if one is densely defined and the following  equality  holds $(Lf,g)_{\mathfrak{H}}=(f,Lg)_{\mathfrak{H}},\,f,g\in   \mathrm{D}  (L).$

    Consider a   sesquilinear form   $ t  [\cdot,\cdot]$ (see \cite{firstab_lit:kato1980} )
defined on a linear manifold  of the Hilbert space $\mathfrak{H}.$   Denote by $   t  [\cdot ]$ the  quadratic form corresponding to the sesquilinear form $ t  [\cdot,\cdot].$
Let   $  \mathfrak{h}=( t + t ^{\ast})/2,\, \mathfrak{k}   =( t - t ^{\ast})/2i$
   be a   real  and    imaginary component     of the   form $  t $ respectively, where $ t^{\ast}[u,v]=t \overline{[v,u]},\;\mathrm{D}(t ^{\ast})=\mathrm{D}(t).$ According to these definitions, we have $
 \mathfrak{h}[\cdot]=\mathrm{Re}\,t[\cdot],\,  \mathfrak{k}[\cdot]=\mathrm{Im}\,t[\cdot].$ Denote by $\tilde{t}$ the  closure   of a   form $t.$  The range of a quadratic form
  $ t [f],\,f\in \mathrm{D}(t),\,\|f\|_{\mathfrak{H}}=1$ is called    {\it range} of the sesquilinear form  $t $ and is denoted by $\Theta(t).$
 A  form $t$ is called    {\it sectorial}    if  its    range  belongs to   a sector  having  a vertex $\iota$  situated at the real axis and a semi-angle $0\leq\theta_{\iota}<\pi/2.$   Suppose   $t$ is a closed sectorial form; then  a linear  manifold  $\mathrm{D}_{0}(t) \subset\mathrm{D} (t)$   is
called    {\it core}  of $t,$ if the restriction   of $t$ to   $\mathrm{D}_{0}(t)$ has the   closure
$t$ (see\cite[p.166]{firstab_lit:kato1980}).   Due to Theorem 2.7 \cite[p.323]{firstab_lit:kato1980}  there exist unique    m-sectorial operators  $T_{t},T_{ \mathfrak{h}} $  associated  with   the  closed sectorial   forms $t,  \mathfrak{h}$   respectively.   The operator  $T_{  \mathfrak{h}} $ is called  a {\it real part} of the operator $T_{t}$ and is denoted by  $Re\, T_{t}.$ Suppose  $L$ is a sectorial densely defined operator and $t[u,v]:=(Lu,v)_{\mathfrak{H}},\,\mathrm{D}(t)=\mathrm{D}(L);$  then
 due to   Theorem 1.27 \cite[p.318]{firstab_lit:kato1980}   the corresponding  form $t$ is   closable, due to
   Theorem 2.7 \cite[p.323]{firstab_lit:kato1980} there exists   a unique m-sectorial operator   $T_{\tilde{t}}$   associated  with  the form $\tilde{t}.$  In accordance with the  definition \cite[p.325]{firstab_lit:kato1980} the    operator $T_{\tilde{t}}$ is called     a {\it Friedrichs extension} of the operator $L.$
Everywhere further,   unless  otherwise  stated,  we   use  notations of the papers   \cite{firstab_lit:1Gohberg1965},  \cite{firstab_lit:kato1980},  \cite{firstab_lit:kipriyanov1960}, \cite{firstab_lit:1kipriyanov1960},
\cite{firstab_lit:samko1987}.\\

\noindent{\bf   Series expansion}\\

In this paragraph, we represent two theorems   valuable from  theoretical  and applied points of view respectively.    The first one is a generalization of the Lidskii method  this  is why  following the the classical approach we divide it into two statements that can be claimed separately. The first  statement establishes a character of the series convergence having a principal meaning within the whole concept. The second statement reflects the name of convergence - Abel-Lidskii since  the latter   can be connected with the definition of the series convergence in the Abel sense, more detailed  information can be found in the monograph by Hardy G.H. \cite{firstab_lit:Hardy}. The second theorem is a valuable  application of the first one, it is based upon   suitable algebraic reasonings having noticed by the author and allowing us to involve   a fractional derivative in the first term. We should note that previously,  a concept of an operator function represented in the second term  was realized  in the paper  \cite{firstab_lit:2kukushkin2022}, where  a case corresponding to  a function   represented by a Laurent series with a polynomial regular part was considered.  Bellow, we consider  a comparatively more difficult  case obviously  related to the infinite regular part of the Laurent series and therefore  requiring  a principally different method of study.

  We can choose a  Jordan basis in   the finite dimensional root subspace $\mathfrak{N}_{q}$ corresponding to the eigenvalue $\mu_{q}$  that consists of Jordan chains of eigenvectors and root vectors  of the operator $B_{q}.$  Each chain has a set of numbers
\begin{equation}\label{12i}
  e_{q_{\xi}},e_{q_{\xi}+1},...,e_{q_{\xi}+k},\,k\in \mathbb{N}_{0},
\end{equation}
  where $e_{q_{\xi}},\,\xi=1,2,...,m $ are the eigenvectors  corresponding   to the  eigenvalue $\mu_{q}.$
Considering the sequence $\{\mu_{q}\}_{1}^{\infty}$ of the eigenvalues of the operator $B$ and choosing a  Jordan basis in each corresponding  space $\mathfrak{N}_{q},$ we can arrange a system of vectors $\{e_{i}\}_{1}^{\infty}$ which we will call a system of the root vectors or following  Lidskii  a system of the major vectors of the operator $B.$
Assume that  $e_{1},e_{2},...,e_{n_{q}}$ is  the Jordan basis in the subspace $\mathfrak{N}_{q}.$  We can prove easily (see \cite[p.14]{firstab_lit:1Lidskii}) that     there exists a  corresponding biorthogonal basis $g_{1},g_{2},...,g_{n_{q}}$ in the subspace $\mathfrak{M}_{q}^{\perp},$ where  $\mathfrak{M}_{q},$ is a subspace  wherein the operator  $B-\mu_{q} I$ is invertible. Using the reasonings \cite{firstab_lit:1kukushkin2021},  we conclude that $\{ g_{i}\}_{1}^{n_{q}}$ consists of the Jordan chains of the operator $B^{\ast}$ which correspond to the Jordan chains  \eqref{12i}, more detailed information can be found in \cite{firstab_lit:1kukushkin2021}.
It is not hard to prove   that  the set  $\{g_{\nu}\}^{n_{j}}_{1},\,j\neq i$  is orthogonal to the set $ \{e_{\nu}\}_{1}^{n_{i}}$ (see \cite{firstab_lit:1kukushkin2021}).  Gathering the sets $\{g_{\nu}\}^{n_{j}}_{1},\,j=1,2,...,$ we can obviously create a biorthogonal system $\{g_{n}\}_{1}^{\infty}$ with respect to the system of the major vectors of the operator $B.$  Consider a sum
\begin{equation}\label{3h}
 \mathcal{A} _{\nu}(\varphi,t)f:= \sum\limits_{q=N_{\nu}+1}^{N_{\nu+1}}\sum\limits_{\xi=1}^{m(q)}\sum\limits_{i=0}^{k(q_{\xi})}e_{q_{\xi}+i}c_{q_{\xi}+i}(t)
\end{equation}
where   $k(q_{\xi})+1$ is a number of elements in the $q_{\xi}$-th Jourdan chain,  $m(q)$ is a geometrical multiplicity of the $q$-th eigenvalue,
\begin{equation}\label{4h}
c_{q_{\xi}+i}(t)=   e^{ -\varphi(\lambda_{q})  t}\sum\limits_{m=0}^{k(q_{\xi})-i}H_{m}(\varphi, \lambda_{q},t)c_{q_{\xi}+i+m},\,i=0,1,2,...,k(q_{\xi}),
\end{equation}
$
c_{q_{\xi}+i}= (f,g_{q_{\xi}+k-i}) /(e_{q_{\xi}+i},g_{q_{\xi}+k-i}),
$
$\lambda_{q}=1/\mu_{q}$ is a characteristic number corresponding to $e_{q_{\xi}},$
$$
H_{m}( \varphi,z,t ):=  \frac{e^{ \varphi(z)  t}}{m!} \cdot\lim\limits_{\zeta\rightarrow 1/z }\frac{d^{m}}{d\zeta^{\,m}}\left\{ e^{-\varphi (\zeta^{-1})t}\right\} ,\;m=0,1,2,...\,,\,.
$$
More detailed information on the considered above   Jordan chains can be found in \cite{firstab_lit:1kukushkin2021}.

\vspace{0.5cm}
\noindent{\bf Previously obtained results}\\

Here, we represent  previously obtained results that will be undergone to  a thorough study since our principal challenge is to obtain an accurate description of the  Schatten-von Neumann class index of a non-selfadjoint operator.

Bellow, we represent  the  conditions of Theorem 1 \cite{kukushkin2021a} that  gives us a description of spectral properties of a non-selfadjoint operator $L$ acting in $\mathfrak{H}.$ \\

 \noindent  $ (\mathrm{H}1) $ There  exists a Hilbert space $\mathfrak{H}_{+}\subset\subset\mathfrak{ H}$ and a linear manifold $\mathfrak{M}$ that is  dense in  $\mathfrak{H}_{+}.$ The operator $L$ is defined on $\mathfrak{M}.$    \\

 \noindent  $( \mathrm{H2} )  \,\left|(Lf,g)_{\mathfrak{H}}\right|\! \leq \! C_{1}\|f\|_{\mathfrak{H}_{+}}\|g\|_{\mathfrak{H}_{+}},\,
      \, \mathrm{Re}(Lf,f)_{\mathfrak{H}}\!\geq\! C_{2}\|f\|^{2}_{\mathfrak{H}_{+}} ,\,f,g\in  \mathfrak{M},\; C_{1},C_{2}>0.
$
\\

Consider  a condition  $\mathfrak{M}\subset \mathrm{D}( W ^{\ast}),$ in this case the real Hermitian component  $\mathcal{H}:=\mathfrak{Re }\,W$ of the operator is defined on $\mathfrak{M},$ the fact is that $\tilde{\mathcal{H}}$ is selfadjoint,    bounded  from bellow (see Lemma  3 \cite{firstab_lit(arXiv non-self)kukushkin2018}), where $H=Re W.$  Hence a corresponding sesquilinear  form (denote this form by $h$) is symmetric and  bounded from bellow also (see Theorem 2.6 \cite[p.323]{firstab_lit:kato1980}). It can be easily shown  that $h\subset   \mathfrak{h},$  but using this fact    we cannot claim in general that $\tilde{\mathcal{H}}\subset H$ (see \cite[p.330]{firstab_lit:kato1980} ). We just have an inclusion   $\tilde{\mathcal{H}}^{1/2}\subset H^{1/2}$     (see \cite[p.332]{firstab_lit:kato1980}). Note that the fact $\tilde{\mathcal{H}}\subset H$ follows from a condition $ \mathrm{D}_{0}(\mathfrak{h})\subset \mathrm{D}(h) $ (see Corollary 2.4 \cite[p.323]{firstab_lit:kato1980}).
 However, it is proved (see proof of Theorem  4 \cite{firstab_lit(arXiv non-self)kukushkin2018}) that relation H2 guaranties that $\tilde{\mathcal{H}}=H.$ Note that the last relation is very useful in applications, since in most concrete cases we can find a concrete form of the operator $\mathcal{H}.$

Further, we consider the Theorem 1 statements separately. Let $W$ be a restriction of the operator  $L$ on the set $\mathfrak{M},$ without loss of generality of reasonings, we assume that $W$ is closed since the conditions  H1, H2  guaranty  that it is closeable, more detailed information in this regard is given in  \cite{kukushkin2021a}.  The following statement establishes estimates for the   Schatten-von Neumann class index, under assumptions $\mathrm{H}1,\mathrm{H}2.$ We have the following classification in terms of the operator order $\mu,$ where it is defined as follows $ \lambda_{n}(R_{H})=O  (n^{-\mu}),\,n\rightarrow\infty.$\\

\noindent $({\bf A})$  If $\mu\leq1,$ then   $\; R_{ \tilde{W} }\in  \mathfrak{S}_{p},\,p>2/\mu.$ If  $\mu>1,$ then $R_{ W }\in  \mathfrak{S}_{1}.$
 Moreover, under assumptions
$\lambda _{ n} \left(H \right)=O  (n^{ \mu}),\,\mu>0,$   the following implication holds $R_{  W}\in\mathfrak{S}_{p},\,p\in [1,\infty),\Rightarrow \mu>1/p.$    \\

 Observe  that the given above classification is far from the exact description of the Schatten-von Neumann class index. However, having analyzed  the above   implications,  we can say that   it makes a prerequisite to establish  a hypotheses $ R_{ W }\in  \mathfrak{S}^{\star} _{1/\mu},\,0<\mu<\infty,$  where
$$
 \mathfrak{S}^{\star}_{\rho}(\mathfrak{H}):=\{T\in\mathfrak{S}_{\rho+\varepsilon},\,T\,  \bar{\in}  \,\mathfrak{S}_{\rho},\,\forall\varepsilon>0 \},
$$
the following narrative is devoted to its verification.

Bellow, we represent the second statement of Theorem 1 \cite{kukushkin2021a}, where the peculiar result related to the asymptotics of the absolute value of the eigenvalue is given\\

\noindent$({\bf B})$     In the case  $\nu(R_{ W })=\infty,\,\mu \neq0,$  the following relation  holds
$$
|\lambda_{n}(R_{W})|=  o\left(n^{-\tau}    \right)\!,\,n\rightarrow \infty,\;0<\tau<\mu.
$$
Note that  this statement  will allow us to arrange brackets in the series that converges in the Abell-Lidskii sense  what would be an advantageous achievement in the later constructed theory.   \\

The following statement will allow us to spread the notion of the convergence in the Abell-Lidskii sense to the whole Hilbert space due to the Matcsaev method\\

\noindent $({\bf C})$   Assume that   $\theta< \pi \mu/2\, ,$    where $\theta$ is   the   semi-angle of the     sector $ \mathfrak{L}_{0}(\theta)\supset \Theta (W).$
Then  the system of   root   vectors  of   $R_{ W }$ is complete in $\mathfrak{H}.$\\
\newpage
\section{Main results}
\vspace{0.5cm}
\noindent{\bf 1. The main refinement of the result  A }\\

The reasonings produced bellow appeals to a compact operator $B$ what represents a most general case in the framework of the decomposition on the root vectors theory, however to obtain more peculiar results we are compelled to deploy some restricting conditions. In this regard we involve hypotheses H1,H2 if it is necessary.
 Consider the Hermitian components
$$
\mathfrak{Re} B:=\frac{B+B^{\ast}}{2},\;\;\mathfrak{Im} B:=\frac{B-B^{\ast}}{2i},
$$
it is clear that they are compact selfadjoint operators, since $B$ is compact and due to the technicalities of the given algebraic constructions.
The following relation can be established by direct calculation
$$
\mathfrak{Re}^{2} \!B+ \mathfrak{Im}^{2} \!B=\frac{B^{\ast}B+BB^{\ast}}{2},
$$
from what follows the inequality
\begin{equation}\label{eq4a}
  B^{\ast}B \leq   2 \mathfrak{Re}^{2} \!B+ 2\mathfrak{Im}^{2}B .
\end{equation}
Having analyzed the latter formula, we see that  it is rather reasonable to think over the opportunity of applying the corollary of the  minimax principle pursuing the aim to estimate the singular numbers of the operator $B.$ For the purpose of implementing the latter concept,   consider the following relation
$
\mathfrak{Re}^{2} \!B \,f_{n}=\lambda^{2}_{n} f_{n},
$
where $f_{n},\lambda_{n}$ are the eigenvectors and the eigenvalues of the operator $\mathfrak{Re}   B$ respectively.  Since the operator $\mathfrak{Re}   B$ is selfadjoint and compact then its set of eigenvalues form a basis in $\overline{\mathrm{R}(\mathfrak{Re}   B)}.$   Assume that there exists a non-zero  eigenvalue of the operator $ \mathfrak{Re}^{2} \!B $ that is different from  $\{\lambda^{2}_{n}\}_{1}^{\infty},$  then, in accordance with the well-known fact of the operator theory,  the corresponding eigenvector is orthogonal to the eigenvectors of the operator $\mathfrak{Re}   B.$  Taking into account the fact that the latter form a basis in $\overline{\mathrm{R}(\mathfrak{Re}   B)},$ we come to the conclusion that the eigenvector does not belong to  $\overline{\mathrm{R}(\mathfrak{Re}   B)}.$   Thus, the obtained contradiction proves the fact
$
\lambda_{n}(\mathfrak{Re}^{2} \!B)=\lambda^{2}_{n}(\mathfrak{Re}   B).
$
Implementing the same reasonings, we obtain
$
\lambda_{n}(\mathfrak{Im}^{2} \!B)=\lambda^{2}_{n}(\mathfrak{Im}   B).
$

Further, we   need a result by Ky Fan \cite{firstab_lit:Fan} that can be found as a corollary of the well-known Allakhverdiyev theorem,  see   Corollary  2.2 \cite{firstab_lit:1Gohberg1965} (Chapter II, $\S$ 2.3), in accordance with which, we have
$$
s_{m+n-1}(\mathfrak{Re}^{2} \!B+ \mathfrak{Im}^{2}B)\leq s_{m}(\mathfrak{Re}^{2} \!B)+s_{n}(\mathfrak{Im}^{2}B),\;\;m,n=1,2,...\,.
$$
Choosing $n=m$ and $n=m+1,$  we obtain respectively
$$
s_{2m-1}(\mathfrak{Re}^{2} \!B+ \mathfrak{Im}^{2}B)\leq s_{m}(\mathfrak{Re}^{2} \!B)+s_{m}(\mathfrak{Im}^{2}B),
$$
$$
\, s_{2m}(\mathfrak{Re}^{2} \!B+ \mathfrak{Im}^{2}B)\leq s_{m}(\mathfrak{Re}^{2} \!B)+s_{m+1}(\mathfrak{Im}^{2}B) \;\;m =1,2,...\,.
$$
At this stage of the reasonings we need involve  the sectorial property $\Theta(B)\subset \mathfrak{L}_{0}(\theta)$ which  gives us $|\mathrm{Im}( Bf,f)|\leq \tan \theta \,\mathrm{Re}(   Bf,f).$  Applying the corollary of the minimax principle to the latter relation, we   get $|\lambda_{n}(\mathfrak{Im}   B)|\leq \tan \theta \, \lambda_{n}(\mathfrak{Re}   B).$ Therefore
$$
s_{2m-1}(\mathfrak{Re}^{2} \!B+ \mathfrak{Im}^{2}B)\leq s_{m}(\mathfrak{Re}^{2} \!B)+s_{m}(\mathfrak{Im}^{2}B)=\sec^{2}\!\theta  s_{m}^{2}(\mathfrak{Re}  B),
$$
$$
\, s_{2m}(\mathfrak{Re}^{2} \!B+ \mathfrak{Im}^{2}B)\leq \sec^{2}\!\theta  s_{m}^{2}(\mathfrak{Re}  B) \;\;m =1,2,...\,.
$$
Applying the minimax principle to the formula \eqref{eq4a}, we get
$$
s_{2m-1}(B)\leq \sqrt{2}\sec \!\theta  s_{m} (\mathfrak{Re}  B),\;\;s_{2m}(B)\leq \sqrt{2}\sec \!\theta  s_{m} (\mathfrak{Re}  B), \;\;m =1,2,...\,  .
$$
This   gives us the upper estimate for the singular values of the operator $B.$

However, to      obtain the lower estimate we need involve Lemma 3.1 \cite[p.336]{firstab_lit:kato1980}, Theorem 3.2 \cite[p.337]{firstab_lit:kato1980}.
 Consider an unbounded   operator $T,\, \Theta(T)\subset \mathfrak{L}_{0}(\theta),$  in accordance with the first representation theorem   \cite[p. 322]{firstab_lit:kato1980}, we can consider its Friedrichs extension the m-sectorial operator $W,$ in its own turn   due to the results \cite[p.337]{firstab_lit:kato1980}, it has a real part $H$ which coincides with the Hermitian  real component if we deal with a bounded operator. Note that by virtue of the sectorial property the operator $H$ is non-negotive. Further, we consider the case $\mathrm{N}(H)=0$ it follows that $\mathrm{N}(H^{\frac{1}{2}})  =0.$ To prove this fact we should note that $\mathrm{def}H=0,$ considering inner product with the element belonging to $\mathrm{N}(H^{\frac{1}{2}})$ we obtain easily the fact  that it must equal to zero.
Having analyzed the proof of Theorem 3.2 \cite[p.337]{firstab_lit:kato1980}, we see that its statement remains true in the modified  form even in the case if we lift the m-accretive condition, thus under the  sectorial condition imposed upon the closed densely defined  operator $T,$ we get the following inclusion
$$
T\subset H^{1/2}(I+iG)H^{1/2},
$$
here   symbol $G$ denotes    a bounded selfadjoint operator in $\mathfrak{H}.$ However, to obtain the asymptotic formula established in Theorem 5 \cite{firstab_lit(arXiv non-self)kukushkin2018} we cannon be satisfied by the made assumptions but require the the existence of the resolvent at the point zero and its compactness. In spite of the fact that   we can  proceed our narrative under the weakened conditions regarding the operator $W$ in comparison with H1,H2, we can   claim that the statement of  Theorem 5 \cite{firstab_lit(arXiv non-self)kukushkin2018} remains true under the assumptions made above, we prefer to deploy H1,H2 what guarantees the conditions we need and at the same time provides a description of the matter under the natural point of view.  Thus, under the made assumptions, using Theorem 3.2 \cite{firstab_lit:kato1980}, we have the following representation
$$
W=H^{1/2}(I+iG)H^{1/2},\;W^{\ast}=H^{1/2}(I-iG)H^{1/2}.
$$
It follows easily from this formula that the Hermitian components of the operator $W$ are defined, we have
$
\mathfrak{Re}W =H,\;\;\mathfrak{Im}W=H^{1/2}GH^{1/2}.
$
Using the decomposition
$
W=\mathfrak{Re}W +i\mathfrak{Im}W,\;W^{\ast}=\mathfrak{Re}W -i\mathfrak{Im}W,
$
we get easily
$$
 \left(  \frac{W^{2}+W^{\ast\,2}}{2}\, f,f\right)=\left\|\mathfrak{Re}W f\right\|^{2}-\left\|\mathfrak{Im}W f\right\|^{2};\,
$$
$$
 \left(  \frac{W^{2}-W^{\ast\,2}}{2i}\, f,f\right)= \frac{1}{2}\left\{\mathfrak{Im}W \,\mathfrak{Re}Wf+\mathfrak{Re}W\, \mathfrak{Im}Wf\right\}.
$$
Using simple reasonings,  we can rewrite the above formulas in terms of Theorem 3.2 \cite{firstab_lit:kato1980}, we have
\begin{equation}\label{eq1}
\mathrm{Re}(W^{2}f,f)=\|Hf\|^{2}-\|H^{1/2}GH^{1/2}f\|^{2},\;\mathrm{Im} (W^{2}f,f)=\mathrm{Re}(H^{1/2}GH^{1/2}f,H^{1/2}f),
$$
$$
\,f\in \mathrm{D}(W^{2}).
\end{equation}
Consider a set of eigenvalues $\{\lambda_{n}\}_{1}^{\infty}$ and a complete system of orthonormal vectors $\{e_{n}\}_{1}^{\infty}$ of the operator $H,$ using the matrix form of the operator $G,$     we have
$$
\|Hf\|^{2}=\sum\limits_{n=1}^{\infty}  |\lambda _{n}|^{2}|f_{n}|^{2},\;\|H^{1/2}GH^{1/2}f\|^{2}=\sum\limits_{n=1}^{\infty} \lambda_{n} \left| \sum\limits_{k=1}^{\infty} b_{nk}\sqrt{\lambda_{k}}f_{k}\right|^{2},
$$
$$
\mathrm{Re}(H^{1/2}GH^{1/2}f,H^{1/2}f)=\mathrm{Re}\left(\sum\limits_{n=1}^{\infty} \lambda_{n}f_{n}  \sum\limits_{k=1}^{\infty} b_{nk}\sqrt{\lambda_{k}}\bar{f_{k}}\right).
$$
Applying the Cauchy-Swarcz inequality, we get
$$
 \|H^{1/2}GH^{1/2}f\|^{2}\leq  \|Hf\|^{2} \sum\limits_{ k,n=1}^{\infty}b^{2}_{ nk}\lambda_{n}/\lambda_{k} ,\;
 |\mathrm{Re}(H^{1/2}GH^{1/2}f,H^{1/2}f)|\leq \|Hf\|^{2} \left(\sum\limits_{ k,n=1}^{\infty}    b^{2}_{nk}/\lambda_{k} \right)^{1/2} .
$$
Therefore, the sectorial condition can be expressed as follows
$$
 \sum\limits_{ k,n=1}^{\infty}b^{2}_{ nk}\lambda_{n}/\lambda_{k}  +  \mathrm{ctg} \theta\left(\sum\limits_{ k,n=1}^{\infty}    b^{2}_{nk}/\lambda_{k} \right)^{1/2}< 1,
$$
where $\theta$ is the  semi-angle of the supposed sector.
Let us find the desired sectorial conditions using the following trivial estimate
$$
   \mathrm{ctg} \theta\left(\sum\limits_{ k,n=1}^{\infty}    b^{2}_{nk}/\lambda_{k} \right)^{1/2} <  \frac{ \mathrm{ctg} \theta}{\lambda_{1}} \left(\sum\limits_{ k,n=1}^{\infty}    b^{2}_{nk}\lambda_{n}/\lambda_{k} \right)^{1/2}  .
$$
Solving the corresponding quadratic equation, we obtain the desired estimate
\begin{equation}\label{eq6d}
 \left(\sum\limits_{ k,n=1}^{\infty}b^{2}_{ nk}\lambda_{n}/\lambda_{k}\right)^{1/2}<\frac{1}{2}\sqrt{\left(\frac{ \mathrm{ctg} \theta}{\lambda_{1}}\right)^{2}+4}-\frac{ \mathrm{ctg} \theta}{\lambda_{1}}.
\end{equation}
  The latter relation gives us a sectorial condition in terms of  matrices of operators $H,G$ for an arbitrary semiangle in the right half-plain. Observe that the right-hand side is less than one what gives us the desired sectorial condition in terms of the absolute  operator norm.
  Note that coefficients  $ b^{2}_{nk}\lambda_{n}/\lambda_{k} $ correspond to the matrices of the operators
$$
H^{1/2}GH^{-1/2}f=\sum\limits_{n=1}^{\infty}  \lambda^{1/2}_{n} e_{n}  \sum\limits_{k=1}^{\infty} b_{nk}\lambda^{-1/2}_{k} f_{k},\;H^{-1/2}GH^{1/2}f=\sum\limits_{n=1}^{\infty}  \lambda^{-1/2}_{n} e_{n}  \sum\limits_{k=1}^{\infty} b_{nk}\lambda^{1/2}_{k} f_{k}   .
$$
Thus, if the absolute operator norm exists, i.e.
$$
\|H^{1/2}GH^{-1/2}\| _{2}:=\left(\sum\limits_{ k,n=1}^{\infty}     b^{2}_{nk}\lambda_{n}/\lambda_{k}\right)^{1/2}<\infty,
$$
then both of them belong to a so-called Hilbert-Schmidt class simultaneously, but it is clear without involving the   absolute norm  since the above operators are adjoint. Taking into account the fact that    the right hand side of \eqref{eq6d} tends to one from bellow when $\theta$ tends to $\pi/2,$ we obtain the condition
\begin{equation}\label{eq7c}
\|H^{1/2}GH^{-1/2}\|_{2}< 1.
\end{equation}
which guarantees  the desired sectorial property. It is remarkable that,   we can write formally the obtained estimate in terms of the Hermitian components of the operator, i.e.
$$
\|\mathfrak{Im}W /\mathfrak{Re}W\|_{2}< 1.
$$
Bellow for a convenient form of writing, we will use a short-hand notation $A:=R_{ W }$ where it is necessary. The next step is to establish the asymptotic formula
\begin{equation}\label{eq8c}
\lambda_{n}\left( \frac{A+A^{\ast 2}}{2}\right)\asymp \lambda^{-1}_{n} \left(\mathfrak{Re}W^{2} \right),\;n\rightarrow\infty,
\end{equation}
however under the additional condition \eqref{eq7c}   we can show easily that conditions  H1,H2  are fulfilled for the operator $W^{2}$ in the reduced form but it is sufficient to  implement  the reasonings of   Theorem 5 \cite{firstab_lit(arXiv non-self)kukushkin2018}. More regular and simple case relates to the additional condition $\mathfrak{M}\subset \mathrm{D}(W^{2}),$ here we can choose the same spaces in  H1, but compelled to repeat the reasonings of  Theorem 5 \cite{firstab_lit(arXiv non-self)kukushkin2018} under the partly  modified hypotheses H2 in the form
$$
\left|(W^{2}f,f)_{\mathfrak{H}}\right|\! \leq \! C_{1}\|f\|^{2}_{\mathfrak{H}_{+}},\,f\in \mathfrak{M}.
$$
It should be noted that the additional condition $\mathfrak{M}\subset \mathrm{D}(W^{2})$ can be lifted, in this case we need deal with the m-sectorial extension of the operator $W^{2},$ at the same time the schemes  of reasonings represented in Theorems 4,5 \cite{firstab_lit(arXiv non-self)kukushkin2018} can be preserved.
Thus, in accordance with the reasonings of  Theorems 4,5 \cite{firstab_lit(arXiv non-self)kukushkin2018} we obtain the desired asymptotic formula \eqref{eq8c}.
 Further, we will use  the following formula obtained due to the  positiveness of the  squared Hermitian imaginary  component of the operator $A,$ we have
$$
\frac{A^{2}+A^{\ast 2}}{2}\leq A^{\ast}A+AA^{\ast}.
$$
Applying the   corollary of the well-known Allakhverdiyev theorem   (Ky Fan  \cite{firstab_lit:Fan}),   see   Corollary  2.2 \cite{firstab_lit:1Gohberg1965} (Chapter II, $\S$ 2.3),   we have
$$
\lambda_{2n}\left(  A^{\ast}A+AA^{\ast}\right)\leq \lambda_{n}(A^{\ast}A)+\lambda_{n}(AA^{\ast}),\;\lambda_{2n-1}\left(  A^{\ast}A+AA^{\ast}\right)\leq \lambda_{n}(A^{\ast}A)+\lambda_{n}(AA^{\ast}),\,n\in \mathbb{N}.
$$
Taking into account the fact $s_{n}(A)=s_{n}(A^{\ast}),$ using the minimax principle, we obtain the estimate
$$
s^{2}_{n}(A)\geq C \lambda_{2n}\left( \frac{A+A^{\ast 2}}{2}\right),\;n\in \mathbb{N},
$$
applying \eqref{eq8c}, we obtain
$$
s^{2}_{n}(A)\geq C \lambda^{-1}_{2n} \left(\mathfrak{Re}W^{2} \right),\;n\in \mathbb{N}.
$$
Here, it is rather reasonable to apply formula \eqref{eq1}, what gives us
$
\mathrm{Re}(W^{2}f,f)\leq(H^{2}f,f) ,
$
what in its own turn, collaboratively with the minimax principle leads  us to
$$
s_{n}(A)\geq C \lambda^{-1}_{2n} \left(H \right),\;n\in \mathbb{N}.
$$
In terms of series, the latter relation gives us
$$
\sum\limits_{n=1}^{\infty}s^{p}_{n}(A)\geq C \sum\limits_{n=1}^{\infty}\lambda^{-p}_{2n} \left(H \right)\geq C \sum\limits_{n=1}^{\infty} \frac{1}{n^{ p\mu}}.
$$
Thus, under the assumptions  $\lambda _{ n} \left(H \right)\leq C n^{\mu},\,\mu>0,\;n\in \mathbb{N}$ the obtained relation  makes it clear that   $R_{W}\in \mathfrak{S}_{p}$ only if when $p>1/\mu.$ Gathering, the above information as well as the previously obtained results, we come to the conclusion that if the order  $\mu$ is defined, then  $R_{W}\in \mathfrak{S}^{\star} _{1/\mu}  $ in the case $\mu\leq 1,$ since firstly, in accordance with the improved $(\mathbf{A}),$ we obtain $R_{W}\in \mathfrak{S}_{p},\,p>1/\mu,$ the latter relation does not imply $R_{W}\bar{\in} \,\mathfrak{S}_{1/\mu}.$ However, collaboratively with   the implication  $R_{  W}\in\mathfrak{S}_{p},\,p\in [1,\infty),\Rightarrow \mu p>1,$ we get the desired result. The case $\mu>1$ gives us, without additional assumptions regarding the operator $G,$ the fact $R_{W}\in \tilde{\mathfrak{S}}_{s},\,s\leq  1/\mu,$ under the additional assumption
$$
\|H^{1/2}GH^{-1/2}\|_{2}< 1,
$$
we have the same formula $R_{W}\in \mathfrak{S}^{\star} _{1/\mu}.$ However, the reasonings given above deserve to be summarized as a theorem, we admit that the peculiar non-standard way of representing the proof may seem unnatural, let alone formal requirements, but notice the  consistence and  breadth of the style, what more can make a paper readable and desirable for the reader?!

\begin{teo} Assume that conditions  H1,H2  hold,  if $\mu\leq 1$ then  $R_{W}\in  \mathfrak{S}^{\star}_{1/\mu},$  if $\mu>1$ then
$R_{W}\in    \mathfrak{S}_{s}\cup\mathfrak{S}^{\star}_{s},\,s\leq  1/\mu,$   if additionally the condition   $
\|H^{1/2}GH^{-1/2}\|_{2}< 1,
$
holds  then $R_{W}\in  \mathfrak{S}^{\star}_{1/\mu}.$
\end{teo}
\begin{proof}
The reasonings corresponding to the theorem proof are represented above in a completely  expanded form.
\end{proof}
Ostensibly,  we come to the remarkable  formula
$
R_{W}\in  \mathfrak{S}^{\star}_{1/\mu}
$
that was previously known for selfadjoint compact operators only.

\vspace{0.5cm}
\noindent{\bf 2. More subtle  asymptotics than one of the power type}\\

In this paragraph, we aim to produce an example of an operator which real part or Hermitian component if it is defined has more subtle asymptotics than one of the power type considered above in detail. Observe the following condition
\begin{equation}\label{eq2}
 (\ln^{1+\kappa}x)'_{\lambda_{n}(H)}  =o(  n^{-\kappa}   ),\; \kappa>0.
\end{equation}
In the paper \cite{firstab_lit:1kukushkin2021}, there was firstly  considered an example of the sequence of the eigenvalues of the real part that satisfy the condition and at the same time the following relation holds
$$
\forall \varepsilon>0:\; n^{-\kappa-\varepsilon}< \frac{1}{\lambda_{n}(H)}  <\frac{C}{n^{\kappa}\ln^{\kappa} \lambda_{n}(H)}  .
$$
Thus, we can contemplate   that  the notion of the order in the used above sense as well as the asymptotics of the power type are spoiled. However, this kind of asymptototics allows us to deploy fully the technicalities given by the Fredholm determinant. The matter that can be turned out to be as a question is "Whether there exists an operator defined analytically which real part satisfies the condition?" It will be the crucial point of our narrative and we approach it from various points of view.
Here,  we demonstrate the mentioned above example having drawn  the reader attention to the fact  that   we lift restrictions on $\kappa$ made in \cite{firstab_lit:1kukushkin2021} dictated by lack of  proposition $(\mathbf{A}).$
\begin{ex}\label{E1} Here we   produce an example of the sequence   $\{\lambda_{n} \}_{1}^{\infty}$  that satisfies   condition \eqref{eq2}, whereas
 $$
\sum\limits_{n=1}^{\infty}\frac{1}{|\lambda_{n}|^{1/\kappa}}=\infty.
$$
\end{ex}
Consider a sequence $\lambda_{n}=n^{\kappa}\ln^{\kappa} (n+q) \cdot \ln^{\kappa}\ln (n+q),\,q>e^{e}-1,\;n=1,2,...,\,.$  Using the integral test for convergence, we can easily see that the previous series is divergent. At the same time substituting, we get
$$
\frac{\ln^{\kappa}\lambda_{n}}{\lambda_{n}} \leq
 \frac{ C\ln^{\kappa}(n+q)  }{n^{\kappa}\ln^{\kappa} (n+q) \cdot \ln^{\kappa}\ln (n+q)}= \frac{ C }{n^{\kappa}   \cdot \ln^{\kappa}\ln (n+q)},
$$
what gives us the fulfilment of the  condition.\\

\noindent{\bf 3. Functional calculus of non-selfadjoint operators}\\

The following consideration  are not being reduced to a trivial case since in consequence with the statement $(\mathbf{C})$ the operator $W$ has an infinite set of eigenvectors. Indeed, since the algebraic multiplicities are finite-dimensional and the system of the root vectors is complete in $\mathfrak{H},$ then the set of the eigenvalues is infinite and we have the claimed statement. Thus,   hypotheses  H1, H2 gives us an opportunity  to avoid a trivial case at least connected with the finite dimension of the eigenvector subspace. However, we can obviously impose this condition directly  without worry of loosing generality. Bellow, we consider a sector $\mathfrak{L}_{0}(\theta_{0},\theta_{1}):=\{z\in \mathbb{C},\, \theta_{0}\leq arg z \leq\theta_{1}\},\,-\pi<\theta_{0}<\theta_{1}<\pi$ and use a short-hand notation $\mathfrak{L}_{0}(\theta):=\mathfrak{L}_{0}(-\theta,\theta).$ We denote by $\lambda_{n},e_{n},\;n\in \mathbb{N}$ the eigenvalues and eigenvectors of the operator $W$ respectively  and define an operator function  as follows
$$
\varphi(W)e_{n}=\lim\limits_{t\rightarrow+0}\frac{1}{2 \pi i}\int\limits_{\vartheta(B)}e^{-\varphi (\lambda) t}\varphi (\lambda)R_{W}(\lambda)e_{n}d\lambda,
$$
where $\varphi$ is an analytic function with a regular growth within  the following  domain   containing  the numerical range of values of the operator argument
 $$
 \vartheta(B):=\left\{\lambda:\;|\lambda|=r>0,\, \theta_{0}-\varepsilon \leq\mathrm{arg} \lambda \leq \theta_{1}+\varepsilon \right\}\cup\left\{\lambda:\;|\lambda|>r,\;
  \mathrm{arg} \lambda =\theta_{0}-\varepsilon ,\,\mathrm{arg} \lambda =\theta_{1}+\varepsilon \right\},
  $$
  $$
  \;\Theta(B) \subset  \mathfrak{L}_{0}(\theta_{0},\theta_{1}),
 $$
where $\varepsilon$ is an arbitrary positive number.
 Bellow, we will show that the definition is correct and coincides with the one given in  \cite{firstab_lit(frac2023)} as well as the  information on the analytic functions with the growth regularity. Using simple reasonings involving properties of the resolvent, Cauchy integral formula e.t.c., we get
$$
\varphi(W)e_{n}=e_{n}\lim\limits_{t\rightarrow+0}\frac{1}{2 \pi i}\int\limits_{\vartheta(B)}e^{-\varphi (\lambda) t}\frac{\varphi (\lambda)}{\lambda_{n}-\lambda} d\lambda=
   \lim\limits_{t\rightarrow+0}e^{-\varphi (\lambda_{n}) t}\varphi(\lambda_{n}) e_{n}=\varphi(\lambda_{n}) e_{n}.
$$
Here, we ought to explain that we managed to pass to the limit considering the contour integrals by virtue of the growth regularity of the function, a complete scheme of  reasonings is represented in \cite{firstab_lit(axi2022)}, \cite{firstab_lit(frac2023)}. Thus, the given above definition corresponds to the closure of the operator function considered in \cite{firstab_lit(frac2023)} (see Lemma 3) and defined on the subset of elements depended on the parameter i.e.
$$
U_{t}e_{n} =\frac{1}{2 \pi i}\int\limits_{\vartheta(B)}e^{-\varphi (\lambda) t} R_{W}(\lambda)e_{n}d\lambda.
$$
Consider the invariant space $\mathfrak{N}$ generated by eigenvectors of the operator, we mentioned above that it is an infinite dimensional space endowed with the stricture of the initial Hilbert space, hence we can consider a restriction of the operator  $R_{W}(\lambda)$ on the space $\mathfrak{N},$ where $\lambda$ does not take values of eigenvalues. By virtue of such an approach and uniqueness of the decomposition on basis vectors in the Hilbert space, we can represent the operator function  defined on elements of $\mathfrak{N}$ in the form of series on eigenvectors
$$
 \varphi (W)f=\sum\limits_{n=1}^{\infty} e_{n} \varphi(\lambda_{n})f_{n},\,f\in \mathrm{D}_{1}(\varphi),
$$
where
$$
\mathrm{D}_{1}(\varphi):=\left\{f\in \mathfrak{N}:\; \sum\limits_{n=1}^{\infty}  | \varphi(\lambda_{n})f_{n}|^{2}<\infty\right\}.
$$
Here, we ought to make a short digression devoted to the closure of the operator $\varphi(W).$ Firstly, consider the following reasonings showing that the operator admits   closure. It was shown in the paper \cite{firstab_lit KRAUNC} that
in accordance with the definition (see (5.6) \cite[p. 165]{firstab_lit:kato1980}),  if there exist  simultaneous   limits  $u^{(j)}_{k}(t)\rightarrow u^{(0)},\; \varphi(W)u^{(j)}_{k}(t)\rightarrow u^{(j)},\,k\rightarrow  \infty,\,t\rightarrow +0,\,j=1,2\,,$  then $u^{(1)}=u^{(2)},$ here the high indexes used to establish the  difference between elements.   Note that the resolvent $R_{W}(\lambda)$ is defined on $\mathfrak{N}$ and admits the following decomposition
$$
R_{W}(\lambda)f=\sum\limits_{n=1}^{\infty}\frac{f_{n}}{\lambda_{n}-\lambda}e_{n},\,f\in \mathfrak{N},
$$
therefore having substituted the latter relation to the formula of the operator function, we get
$$
\varphi(W)f=  \lim\limits_{t\rightarrow+0}\frac{1}{2 \pi i}\int\limits_{\vartheta(W)}e^{-\varphi (\lambda) t}\varphi (\lambda)\sum\limits_{n=1}^{\infty}\frac{f_{n}e_{n}}{\lambda_{n}-\lambda} d\lambda=
$$
$$
=\lim\limits_{t\rightarrow+0}\frac{1}{2 \pi i}\sum\limits_{n=1}^{\infty} f_{n}e_{n} \int\limits_{\vartheta(W)}\frac{ e^{-\varphi (\lambda) t}\varphi (\lambda)}{\lambda_{n}-\lambda} d\lambda=
 \lim\limits_{t\rightarrow+0}\sum\limits_{n=1}^{\infty}f_{n}e_{n} \!\mathop{\operatorname{res}}\limits_{z=\lambda_{n}}  \{  e^{-\varphi (z) t}\varphi (z)\}=
 $$
 $$
 =\lim\limits_{t\rightarrow+0}\sum\limits_{n=1}^{\infty}f_{n}e_{n}     e^{-\varphi (\lambda_{n}) t}\varphi (\lambda_{n})=\sum\limits_{n=1}^{\infty}f_{n}e_{n}     \varphi (\lambda_{n}),\,f\in \mathrm{D}_{1}(\varphi).
$$
The justification of an opportunity to  integrate  the series termwise  is based upon the fact that the latter series is convergent in the sense of the norm. The passage to the limit when $t$ tends to zero is justified by the same fact. Thus, we have obtained the equivalence of definitions, compare with the one given in \cite{firstab_lit(frac2023)}). It is clear that considering the set $U_{t}f,\;f\in \mathfrak{N},\;t>0,$ we can expand the domain of definition of the operator function $\varphi$ at the same time the extension remains closed as one can easily see.

Suppose $\mathfrak{H}:=\mathfrak{N}$ and let us construct a space $\mathfrak{H}_{+}$ satisfying the condition of compact embedding $\mathfrak{H}\subset\subset \mathfrak{H}_{+}$ and suitable for spreading  the condition H2 upon the operator $\varphi(W).$ For this purpose, define
$$
\mathfrak{H}_{+}:=\left\{f\in  \mathfrak{N}:\;\|f\|^{2}_{\mathfrak{H}_{+}}= \sum\limits_{n=1}^{\infty}  | \varphi(\lambda_{n})||f_{n}|^{2}<\infty\right\},
$$
and let us prove the fact $\mathfrak{H}\subset\subset \mathfrak{H}_{+}.$ The idea of the proof is based on the application of the criterion of compactness in Banach spaces, let us involve the operator $B: \mathfrak{H}\rightarrow \mathfrak{H}$ defined as follows
$$
Bf=\sum\limits_{n=1}^{\infty}  | \varphi(\lambda_{n})|^{-1/2} f_{n} e_{n},
$$
here we are assuming without loss of generality that  $|\varphi(\lambda_{n})|\uparrow \infty.$ Note  that in any case, we can rearrange the sequence in the required way  having imposed  a condition of the growth regularity upon the operator function. Observe that if $\|f\|<K=\mathrm{const},$ then
$$
\|R_{k}Bf\|=\sum\limits_{n=k}^{\infty}  | \varphi(\lambda_{n})|^{-1 } |f_{n}|^{2}\leq\frac{\|f\|^{2}}{|\varphi (\lambda_{k})|} <\frac{K^{2}}{|\varphi (\lambda_{k})|},\,k\in \mathbb{N}.
$$
Therefore, in accordance with the compactness criterion in Banach spaces  the operator $B$ is compact. Now, consider a set bounded in the sense of the norm $\mathfrak{H}_{+},$ we will denote its elements by $f,$ thus in accordance with the above, we have
$$
\sum\limits_{n=1}^{\infty}  | \varphi(\lambda_{n})||f_{n}|^{2}<C.
$$
It is clear that the element $g:=\{|\varphi(\lambda_{n})|^{1/2} f_{n}\}_{1}^{\infty}$ belongs to $\mathfrak{H}$ and the set of elements from $\mathfrak{H}$ corresponding to the bounded set of elements from $\mathfrak{H}_{+}$ is bounded also. This is why the operator $B$ image of the set of elements $g$ is compact, but we have
$ Bg=f.$ The latter relation proves the fact that the set of elements $f$ bounded in the sense of the norm $\mathfrak{H}_{+}$ is a bounded set in the sense of the norm $\mathfrak{H}.$ Thus, we have established the fulfilment of condition H1, it is obvious that we can choose a span of $\{e_{n}\}_{1}^{\infty}$ as the mentioned   linear manifold $\mathfrak{M}.$

The verification of the first   relation of  H2 is implemented due to direct application of the Cauchy-Schwarz inequality, we left it to the reader. Let us verify the fulfilment of the  second condition, here we need impose a sectorial condition upon the analytic  function $\varphi$ i.e. it should preserve in the open right-half plain the closed sector belonging to the latter, then we have

$$
\mathrm{Re}\sum\limits_{n=1}^{\infty}    \varphi(\lambda_{n}) |f_{n}|^{2}\geq \sec \psi  \sum\limits_{n=1}^{\infty}    |\varphi(\lambda_{n})| |f_{n}|^{2} ,
$$
where $\psi$ is the half-angle of the sector contenting  the image of the analytic function $\varphi.$  Thus, the condition  H2 is fulfilled. Observe that
$$
\mathfrak{Re}\varphi(W)f=\sum\limits_{n=1}^{\infty}  \mathrm{Re}\varphi(\lambda_{n})f_{n}e_{n},\;f\in \mathrm{D}_{1}(\varphi).
$$
Therefore, implementing obvious reasonings, we get
\begin{equation}\label{eq3}
\lambda_{n}\{\mathfrak{Re}\varphi(W)\}=\mathrm{Re}\varphi(\lambda_{n}  ),\,n\in \mathbb{N}.
\end{equation}
The fact   that there does not exist additional eigenvalues of the operator  $\mathfrak{Re}\varphi(W)$ follows from the fact that $\{e_{n}\}_{1}^{\infty}$ forms a basis in $\mathfrak{N}$ and can be established easily by implementing the scheme of reasonings applied above to the similar cases.

Note that condition \eqref{eq2} plays the distinguished role in the refinement  of the Lidskii results \cite{firstab_lit:1kukushkin2021} since  it guaranties   the equality of the convergence exponent and   the order of summation in the Abell-Lidskii sense. It is rather reasonable to expect that we are highly motivated to produce a concrete example of the operator satisfying the condition \eqref{eq2} for if we found it then it would stress the significance and novelty of the papers
 \cite{firstab_lit:1kukushkin2018,firstab_lit(arXiv non-self)kukushkin2018,kukushkin2019,firstab_lit:1kukushkin2021,kukushkin2021a,firstab_lit:2kukushkin2022,firstab_lit(axi2022),firstab_lit(frac2023)}.
Having been inspired  by the idea, we can use the function considered in Example \ref{E1} as an indicator to find the desired operator function. Thus,  to satisfy condition $\eqref{eq2},$ having taken into account relation \eqref{eq3}, we can impose the following condition: for sufficiently large numbers $n\in \mathbb{N},$  the following   relation holds
\begin{equation}\label{eq4}
  C_{1}<\frac{(n \ln n\cdot \ln\ln n)^{\kappa}}{\mathrm{Re}  \varphi(\lambda_{n}  ) } <C_{2},\;\kappa>0.
\end{equation}
Consider   a function   $\psi(z)=  z^{\xi}\ln z \cdot \ln\ln z ,\,0<\xi\leq1$ in the sector $|\arg z|\leq \theta,$ where for the simplicity of reasonings the   branch of the power function has been chosen so that it acts onto the sector and  we have chosen the branch of the logarithmic function  corresponding to the value   $\phi:=\arg z .$
Let us find the real and imaginary parts of the function $\psi(z),$ we have
$$
z^{\xi}\ln z \cdot \ln\ln z= |z|^{\xi}e^{i\xi\phi} \left(\ln|z|+i\phi\right)\left(a+i \arctan \frac{\phi}{\ln|z|}\right),
$$
where, we denote $a:=\ln|\ln|z|+i\phi| .$ Thus, separating the real and imaginary parts of the function $\psi(z),$ we have
$$
|z|^{\xi}e^{i\xi\phi}\left\{a\ln|z|-\phi \arctan \frac{\phi}{\ln|z|}+i\left( a\phi+\ln|z|\arctan \frac{\phi}{\ln|z|}  \right) \right\}=
$$
$$
|z|^{\xi}\cos \xi\phi\left(a\ln|z|-\phi \arctan \frac{\phi}{\ln|z|} \right)-|z|^{\xi}\sin \xi\phi\left( a\phi+\ln|z|\arctan \frac{\phi}{\ln|z|} \right)+
$$
$$
+i\left\{|z|^{\xi}\sin \xi\phi\left( a\ln|z|-\phi \arctan \frac{\phi}{\ln|z|}\right)+ |z|^{\xi}\cos \xi\phi\left( a\phi+\ln|z|\arctan \frac{\phi}{\ln|z|} \right)\right\}.
$$
It leads us to the following estimate
$$
\frac{\mathrm{Re }\,\psi(z)}{\mathrm{Im} \,\psi(z)} \leq
\frac{\tan \xi\theta\left( a\ln|z|-\phi \arctan  \ln^{-1}\!|z|^{1/\phi}  \right)+  \left( a\phi+\ln|z|\arctan  \ln^{-1}\!|z|^{1/\phi} \right)}{ \left(a\ln|z|-\phi \arctan  \ln^{-1}\!|z|^{1/\phi} \right)-\tan \xi\theta\left( a\phi+\ln|z|\arctan  \ln^{-1}\!|z|^{1/\phi} \right)}.
$$
The latter   gives us an opportunity to claim that for arbitrary $\varepsilon>0,$ there exists $R(\varepsilon)$ such that the following estimate holds
$$
\frac{\mathrm{Re }\psi(z)}{\mathrm{Im} \psi(z)}< (1+\varepsilon)\tan \xi\theta,\,|z|>R(\varepsilon).
$$
Apparently, we can claim that the function $\psi(z)$ nearly preserves the sector $|\arg z|\leq \theta,$ what is completely sufficient for our reasonings for we are dealing with the  neighborhood  of the infinitely distant point.   Let us calculate the absolute value, we have
$$
|\psi(z)|^{2}=|z|^{2\xi}\left\{ (a\ln|z|)^{2}+  ( \phi \arctan  \ln^{-1}\!|z|^{1/\phi}  )^{2}+   ( a\phi   )^{2}+ (  \ln|z|\arctan  \ln^{-1}\!|z|^{1/\phi}  )^{2} \right\}=
$$
$$
= |z|^{2\xi}\left\{  a ^{2}(\ln^{2}|z|+\phi^{2})+   (\ln^{2}|z|+\phi^{2}) \arctan^{2}  \ln^{-1}\!|z|^{1/\phi}     \right\}=
$$
$$
 = |z|^{2\xi}   (\ln^{2}|z|+\phi^{2})(a ^{2}+\arctan^{2}  \ln^{-1}\!|z|^{1/\phi}),
$$
  the latter relation establishes the growth regularity of the function $\psi(z).$
Note that   we obtain the following formula
$$
\mathrm{Re}\,\psi^{\kappa} (z)=|z|^{\xi\kappa}   \left\{\ln^{2}|z|+\phi^{2})^{\kappa/2}(\ln^{2}|\ln|z|+i\phi|+\arctan^{2}  \ln^{-1}\!|z|^{1/\phi})^{\kappa/2}
 \cos(\kappa\arg  \psi )\right\},\;\kappa\in(0,1),
$$
here we should note the distinguished fact that can be  obtained easily from the  above  $\arg\psi (z)\rightarrow \xi\arg z,\,|z|\rightarrow\infty.$ Hence
$$
   \mathrm{Re}\,\psi^{\kappa} (z) \sim  \psi^{\kappa}(|z|)   \cos(\xi\kappa \arg  z),
$$
from what follows the condition $\xi\!\kappa\, \theta<\pi/2$ which guarantees positiveness of $\mathrm{Re}\,\psi^{\kappa} (z)$  for sufficiently large   value   $|z|.$
Here, we should make a short narrative digression and remind that we pursue a rather particular aim to produce an example of an operator so we are free in some sense to choose an operator as an object for our needs. At the same time the given above reasonings    origin  from the fundamental scheme  and as a result allow to construct a fundamental theory. Define the function $\varphi(z):=\psi^{\kappa}(z)$ and consider the operator $W$  such that  $|\lambda_{n}(W )|\asymp n^{1/\xi}.$   Note that  a simple observation gives us the fact
 $\mathrm{D}_{1}(W)\subset\mathrm{D}_{1}(\varphi),$ where
$$
\mathrm{D}_{1}(W):=\left\{f\in \mathfrak{N}:\; \sum\limits_{n=1}^{\infty}  |  \lambda_{n} f_{n}|^{2}<\infty\right\}.
$$
 Therefore, the operator $W$ has a normal restriction $W_{1}$ on $\mathrm{D}_{1}(W)$ and we have $ s_{n}(W_{1})=|\lambda_{n}(W)|,\,n\in \mathbb{N}.$
At the same time, due to the sectorial property of the operator $W,$ we have $|\lambda_{n}(W)|\asymp \mathrm{Re} \lambda_{n}(W),$ taking into account the fact
$
 \mathrm{Re} \lambda_{n}(W)=\lambda_{n}(\mathfrak{Re} W_{1} )  ,\,n\in \mathbb{N},
$
which can be proved due  to the analogous  reasonings   corresponding to   relation \eqref{eq3}, we get finally
 $$
  \mu(\mathfrak{Re} W_{1})=1/\xi .
 $$ Note that the inverse chain of reasonings is obviously correct,  thus  we obtain a description in terms of asymptotics of the  real component  eigenvalues. Eventually, the given above reasonings lead us to the conclusion that     relation \eqref{eq4} holds and we obtain the desired operator with more subtle asymptotics of the real  component than the asymptotics of the power type. We are pleased to represent it in the refined forms due to the ordinary  properties of the exponential function
$$
\varphi(W)f=\lim\limits_{t\rightarrow+0}\frac{1}{2 \pi i}\int\limits_{\vartheta(W)}
 \left(\ln \lambda\right)^{ -t \varphi(\lambda)/\ln\ln\lambda }
\varphi(\lambda) \,R_{W}(\lambda)fd\lambda,\,f\in \mathrm{D}_{1}(W),
$$
we have
$$
\varphi(W)f=\lim\limits_{t\rightarrow+0}\frac{1}{2 \pi i}\int\limits_{\vartheta(W)} \lambda^{ -t \varphi(\lambda)/   \ln  \lambda  }
\varphi(\lambda) \,R_{W}(\lambda)fd\lambda,\,f\in \mathrm{D}_{1}(W).
$$
Now,  observe  benefits and  disadvantages of the idea to involve the restriction of the operator $\varphi(W)$ on the space $\mathfrak{N}.$ Apparently, the main disadvantage is the requirement in accordance with which we need deal with the expansion of the adjoint operator, since we have  $W_{1}\subset W \Rightarrow W^{\ast}\subset W^{\ast}_{1}.$ This follows that the orders of $\mathfrak{Re} W$ and $\mathfrak{Re} W_{1}$ may differentiate, what brings us the essential inconvenience if we want to provide a description in terms of the class $\mathfrak{S}^{\star}_{1/\mu}.$\\

\noindent{\bf 4. Domain of definition of the operator function including well-known operators}\\

However,   let us consider a remarkably showing case  $\mu=1$ corresponding  to the so-called quasi-trace  operator class $\mathfrak{S}^{\star}_{1},$   consider the  operator $
L:=a_{2}\Delta+a_{0},
$ with a constant  complex coefficients   acting in $L_{2}(\Omega),$ here $\Omega\subset \mathbb{E}^{2}$ is a bounded domain with a sufficiently smooth boundary. It is clear that the operator is normal, moreover it is well-known fact that under the conditions  imposed upon $\Omega,$ we have the following relation for the operator order:  $\mu =m/n,$ where $m$ is the highest derivative and $n$ is  a dimensional of the Euclidian space. Thus, in this case we have $\mu(\mathfrak{Re}L)=1.$ It is not hard to prove that $\mathrm{N}(L^{\ast})=0,$ hence $\overline{\mathrm{R}(L)}=L_{2}(\Omega).$ Using this fact, we can claim that $\mathfrak{Re}L$ and $L$ have the same eigenfunctions, the proof is given above for the analogous case, hence we can implement the above scheme of reasonings and consider the operator function $\varphi(L)$ for which condition \eqref{eq4} and as a result condition \eqref{eq2} holds. It is remarkable that   despite of  the fact that the hypotheses $\mathrm{H}1,\mathrm{H}2$ hold for the operator $L$ in the natural  way, the explanation  is left to the reader, we can use the benefits of the scheme introduced above to construct the required pair of Hilbert spaces.

The next case, within  the scale  of most important ones, appeals to the so-called quasi-Hilbert-Schmidt class $\mathfrak{S}^{\star}_{2}.$  In this regard, let us consider the Sturm-Liouville operator $L,$  where the corresponding Euclidian space is one-dimensional. Thus, we obtain the order $\mu(\mathfrak{Re}L)=2$ and this is why have been compelled to choose $\xi=1/2,$ the general scheme of the reasonings are absolutely analogous to the previous case. However, in order to make a clear demonstration let us consider a simplest case corresponding to the operator
$
Lu:=-u'',\; u(0)=u(\pi  )=0,
$
acting in $L_{2}(I),\,I:=(0,\pi ).$ Recall that the general solution of the homogeneous equation $-u''-\lambda u=0,\,\lambda\in \mathbb{R}^{+}$ is given by the following formula
$
u=C_{1}\sin \sqrt{\lambda}x+C_{2}\cos\sqrt{\lambda}x.
$
Using the initial conditions, we find $C_{2}=0,\,\sin\sqrt{\lambda}\,\pi=0.$ Hence $\lambda_{n}=n^{2},\, e_{n}(x)=\sin n  x,\;n\in \mathbb{N}$ are non-zero eigenvalues and  eigenvectors respectively. Note that the closure of the linear  span of the functions  $\sin n  x,\;n\in \mathbb{N}$ gives us the Hilbert space $L_{2}(I).$   The given above theory tells us that the operator function $\varphi(L)$ is defined on the $\mathrm{D}_{1}(L)$ which, in this case, coincides with the functions having a fourier coefficients with a  sufficiently large decrease so that the following series converges in the sense of $L_{2}(I)$ norm, i.e.
$$
\sum\limits_{n=1}^{\infty}n^{2}f_{n}e_{n}(x)\in L_{2}(I).
$$
In addition, we should add that we can consider an arbitrary non-selfadjoint differential and fractional differential operators assuming that the functional space is defined on the bounded  domain of an Euclidian space with a sufficiently   smooth boundary  (regular operators). In most of such  cases the minimax principle can be applied and we can obtain the asyptotics of the eigenvalues of the Hermitian real component, here we can referee a detailed description represented in the monograph by Rozenblyum  G.V. \cite{firstab_lit:Rosenblum}. The case corresponding to an unbounded domain (irregular operator) is also possible for study, in this regard  the Fefferman concept   covers such  problems \cite[p.47]{firstab_lit:Rosenblum}.
 The given above  theoretical results  can be applied to   the operator class  and we can construct in each case a corresponding operator function representing to the reader    the example of an operator with a more subtle asymptotics of the Hermitian real component eigenvalues than one of the power type. Bellow, we represent well-known non-selfadjoint operators which can be considered as an operator argument.\\

\noindent{\it 1. Kolmogorov operator}\\

The relevance of the considered operator is justified by resent results by Goldstein et al.   \cite{firstab_lit:Goldstein} where the following operator has been undergone to the rapt attention
$$
L=\Delta+\frac{\nabla\rho}{\rho}\cdot\nabla,
$$
here $\rho$ is a probability density on $\mathbb{R}^{N}$ satisfying $\rho\in C^{1+\alpha}_{lok}(\mathbb{R}^{N})$
  for some
$\alpha\in (0, 1),\; \rho(x) > 0$ for all $x\in \mathbb{R}^{N}.$

Apparently,  the results \cite{firstab_lit(arXiv non-self)kukushkin2018}, \cite{kukushkin2021a}, \cite{firstab_lit:1kukushkin2021} can be applied to the operator after an insignificant modification. A couple of words on the difficulties appearing while we study the operator composition.  Superficially, the problem  looks pretty well  but it is not so for the inverse operator (one need prove that it is a resolvent)  is a composition of an unbounded operator and a resolvent of the operator $W,$  indeed  since $R_{W}W=I,$ then formally, we have
$
L^{-1}f= R_{W}\rho f.
$
Most likely,   the general theory created in the papers \cite{firstab_lit(arXiv non-self)kukushkin2018}, \cite{kukushkin2021a} can be adopted to some operator composition but it is  a tremendous work. Instead of that, in the paper  \cite{firstab_lit(frac2023)}   we succeed   to  find a suitable pair of Hilbert spaces what allows us to apply theoretical results.\\

\noindent{\it 2. The linear combination of the second order differential operator and the   Kipriyanov operator}\\

  Consider a linear combination of the uniformly elliptic operator, which is written in the divergence form, and
  a composition of a   fractional integro-differential  operator, where the fractional  differential operator is understood as the adjoint  operator  regarding  the Kipriyanov operator  (see  \cite{kukushkin2019,firstab_lit:kipriyanov1960,firstab_lit:1kipriyanov1960})
\begin{equation*}
 L  :=-  \mathcal{T}  \, +\mathfrak{I}^{\sigma}_{ 0+}\phi\, \mathfrak{D}  ^{ \beta  }_{d-},
\; \sigma\in[0,1) ,
 $$
 $$
   \mathrm{D}( L )  =H^{2}(\Omega)\cap H^{1}_{0}( \Omega ),
  \end{equation*}
where
$\,\mathcal{T}:=D_{j} ( a^{ij} D_{i}\cdot),\,i,j=1,2,...,m,$
under    the following  assumptions regarding        coefficients
\begin{equation} \label{12}
     a^{ij}(Q) \in C^{2}(\bar{\Omega}),\, \mathrm{Re} a^{ij}\xi _{i}  \xi _{j}  \geq   \gamma_{a}  |\xi|^{2} ,\,  \gamma_{a}  >0,\,\mathrm{Im }a^{ij}=0 \;(m\geq2),\,
 \phi\in L_{\infty}(\Omega).
\end{equation}
 Note that in the one-dimensional case, the operator $\mathfrak{I}^{\sigma }_{ 0+} \phi\, \mathfrak{D}  ^{ \beta  }_{d-}$ is reduced to   a  weighted fractional integro-differential operator  composition, which was studied properly  by many researchers
    \cite{firstab_lit:2Dim-Kir,firstab_lit:15Erdelyi,firstab_lit:9McBride,firstab_lit:nakh2003}, more detailed historical review  see  in \cite[p.175]{firstab_lit:samko1987}.        \\

\noindent{\it 3. The linear combination of the second order differential operator and the  Riesz potential}\\

Consider a   space $L_{2}(\Omega),\,\Omega:=(-\infty,\infty)$    and the Riesz potential
$$
I^{\beta}f(x)=B_{\beta}\int\limits_{-\infty}^{\infty}f (s)|s-x|^{\beta-1} ds,\,B_{\beta}=\frac{1}{2\Gamma(\beta)  \cos  (\beta \pi / 2)   },\,\beta\in (0,1),
$$
where $f$ is in $L_{p}(\Omega),\,1\leq p<1/\beta.$
It is  obvious that
$
I^{\beta}f= B_{\beta}\Gamma(\beta) (I^{\beta}_{+}f+I^{\beta}_{-}f),
$
where
$$
I^{\beta}_{\pm}f(x)=\frac{1}{\Gamma(\beta)}\int\limits_{0}^{\infty}f (s\mp x) s ^{\beta-1} ds,
$$
these operators are known as fractional integrals on the  whole  real axis   (see \cite[p.94]{firstab_lit:samko1987}). Assume that the following  condition holds
 $ \sigma/2 + 3/4<\beta<1 ,$ where $\sigma$ is a non-negative  constant.
 Following the idea of the   monograph \cite[p.176]{firstab_lit:samko1987},
 consider a sum of a differential operator and  a composition of    fractional integro-differential operators
\begin{equation*}
 W   :=  D^{2} a  D^{2}  +   I^{\sigma}_{+}\,\xi \,I^{2(1-\beta)}D^{2}+\delta I,\;\mathrm{D}(W)=C^{\infty}_{0}(\Omega),
\end{equation*}
where
$
\xi(x)\in L_{\infty}(\Omega),\, a(x)\in L_{\infty}(\Omega)\cap C^{ 2  }( \Omega ),\, \mathrm{Re}\,a(x) >\gamma_{a}(1+|x|)^{5}.
$

\vspace{0.5cm}

\noindent{\it 4. The perturbation of the  difference operator with the artificially constructed operator }\\

Consider a   space $L_{2}(\Omega),\,\Omega:=(-\infty,\infty),$   define a family of operators
$$
T_{t}f(x):=e^{-c t}\sum\limits_{k=0}^{\infty}\frac{(c \,t)^{k}}{k!}f(x-d\mu),\,f\in L_{2}(\Omega),\;c,d>0,\; t\geq0,
$$
where convergence is understood in the sense of $L_{2}(\Omega)$ norm. In accordance with the Lemma 6 \cite{kukushkin2021a}, we know that  $T_{t}$ is a $C_{0}$ semigroup of contractions, the corresponding  infinitesimal generator and its adjoint operator are defined by the following expressions
 $$
Yf(x)=c[f(x)-f(x-d)],\,Y^{\ast}f(x)=c[f(x)-f(x+d)],\,f\in L_{2}(\Omega).
$$
Let us find a representation for fractional powers of the operator $Y.$ Using formula  (45) \cite{kukushkin2021a}, we get
$$
   Y^{\beta}f=\sum\limits_{k=0}^{\infty}M_{k}f(x-kd), \,f\in L_{2}(\Omega),
   \,M_{k}= -\frac{\beta\Gamma(k -\beta)}{k!\Gamma(1 -\beta)}c^{\,\beta},\,\beta\in (0,1).
   $$
Consider the operator
$$
L:= Y^{\ast}\!a Y+b Y^{\beta}+ Q^{\ast}N Q,
$$
where   $a,b\in L_{\infty}(\Omega),\;Q$ is a   closed operator acting in $L_{2}(\Omega),\,Q^{-1}\in \mathfrak{S}_{\!\infty}(L_{2}),$ the operator $N$ is strictly accretive, bounded, $\mathrm{R}(Q)\subset \mathrm{D}(N).$ Note that   Theorem 14 \cite{kukushkin2021a}  gives us a compleat substantiation of the  given above theoretical approach application.\\

\noindent{\bf 5. General  approach in constructing an  operator with a more subtle asymptotics }\\

Recall that an arbitrary compact operator $B$ can be represented as a series on the basis vectors due to the so-called polar decomposition $B=U|B|,$ where $U$ is a concrete unitary operator, $|B|:=(B^{\ast}B)^{1/2},$ i.e.    using the system of eigenvectors  $\{e_{n}\}_{1}^{\infty},$ we have
\begin{equation}\label{eq6}
Bf=\sum\limits_{n=1}^{\infty}s_{n}(B)(f,e_{n})g_{n},
\end{equation}
where $ e_{n}, s_{n} $ are   the eigenvectors and eigenvalues of the operator $|B|$ respectively, $g_{n}=Ue_{n}.$ It is clear that the letter elements form a an orthonormal system due to the major property of the unitary operator. The main concept that lies in the framework of our consideration relates to the following statement.\\

\noindent $(\mathrm{S} 1)$ {\it Under the assumptions $B\in \mathfrak{S}^{\star}_{\rho},\,0<\rho<\infty,\,\Theta(B) \subset   \mathfrak{L}_{0}(\theta),$  a sequence of natural numbers $\{N_{\nu}\}_{0}^{\infty}$ can be chosen so that
 \begin{equation*}\label{eq6a}
 \frac{1}{2\pi i}\int\limits_{\vartheta(B)}e^{-\lambda^{\alpha}t}B(I-\lambda B)^{-1}f d \lambda =\sum\limits_{\nu=0}^{\infty} \mathcal{A}_{\nu}(\lambda^{\alpha},t)f,
$$
where
$\vartheta(B):=\left\{\lambda:\;|\lambda|=r>0,\,|\mathrm{arg} \lambda|\leq \theta+\varepsilon\right\}\cup\left\{\lambda:\;|\lambda|>r,\; |\mathrm{arg} \lambda|=\theta+\varepsilon\right\},\,\alpha$ is a positive number.
 Moreover
$$
\sum\limits_{\nu=0}^{\infty}\left\|\mathcal{A}_{\nu}(\psi,t)f\right\|<\infty.
\end{equation*} }

 Note that in accordance with the Lemma 3 \cite{firstab_lit:1kukushkin2021}, we have
\begin{equation}\label{eq7}
\ln r\frac{n(r)}{r^{\rho }}\rightarrow0,\Longrightarrow   \ln r \left(\int\limits_{0}^{r}\frac{n(t)}{t^{p+1 }}dt+
r\int\limits_{r}^{\infty}\frac{n(t)}{t^{p+2  }}dt\right)r^{p-\rho  } \rightarrow 0,\,r\rightarrow\infty,
\end{equation}
where $\rho$ is a convergence exponent (non integer) defined as follows  $\rho:=\inf \lambda,$
$$
 \sum\limits_{n=1}^{\infty}s_{n}^{\,\lambda}(B)<\infty,
$$
 $n(t)$ is a counting function  corresponding to the singular numbers of the operator $B,$ the number $\lambda=p+1$ is the smallest natural number for which the latter series is convergent. To produce an operator class with more subtle asymptotics, we can deal with representation \eqref{eq6} directly imposing conditions upon the singular numbers instead of considering Hermitian real  component. Assume that the sequence of singular numbers has a non integer  convergent exponent  $\rho$ and  consider the following condition
$$
 (\ln^{1+1/\rho}x)'_{s^{-1}_{m}(B)}  =o(  m^{-1/\rho}   ).
$$
This gives us
$$
\frac{m\ln s^{-1}_{m}(B)}{s^{-\rho}_{m}(B)}\leq C\cdot  \alpha_{m},\;\alpha_{m}\rightarrow 0,\,m\rightarrow\infty.
$$
Taking into account the facts
$
n(s^{-1}_{m})=m;\,\;n(r)= n(s^{-1}_{m}),\,s^{-1}_{m}<r<s^{-1}_{m+1},
$
using the monotonous property of the functions, we get
\begin{equation}\label{eq11a}
\ln r\frac{n(r)}{r^{\rho}}  <  C\cdot  \alpha_{m},\;s^{-1}_{m}<r<s^{-1}_{m+1},
\end{equation}
i.e. we obtain the following implication

$$
 (\ln^{1+1/\rho}x)'_{s^{-1}_{m}(B)}  =o(  m^{-1/\rho}   ),\Longrightarrow \left(\ln r\frac{n(r)}{r^{\rho }}\rightarrow0\right),
$$
therefore the premise in \eqref{eq7} holds  and as a result the consequence holds. Absolutely analogously, we can obtain the implication
\begin{equation}\label{eq11}
  s^{-1}_{m}(B)   =o(  m^{-1/\rho}   ), \Longrightarrow \left(\frac{n(r)}{r^{\rho}}\rightarrow 0\right),
\end{equation}
here the corresponding  example is given by $s_{m}=(m \ln m)^{-1/\rho}$ and the reader can verify directly that $B\in \mathfrak{S}_{\rho}^{\star}.$   Using these facts, we can reformulate Theorem 2, 4 \cite{firstab_lit:1kukushkin2021} for an artificially constructed compact operator. In this paper, we represent   modified  proofs  since there are some difficulties that require refinement, moreover, we produce the variant  of the proof of Theorem 4 \cite{firstab_lit:1kukushkin2021} corresponding to the case      $\Theta(B) \subset   \mathfrak{L}_{0}(\theta),\,\theta<  \pi/2\alpha,$ that is not given in \cite{firstab_lit:1kukushkin2021}.

\begin{teo}\label{T2} Assume that $B$ is a compact operator, $\Theta(B) \subset   \mathfrak{L}_{0}(\theta),\,\theta< \pi/2\alpha ,\;B\in \mathfrak{S}^{\star} _{\alpha},$
$$
s_{n}(B)=o(n^{-1/\alpha}),
$$
where $\alpha$ is positive non integer.
Then the statement of   Theorem 2 \cite{firstab_lit:1kukushkin2021} remains true, i.e. statement S1 holds.
\end{teo}
\begin{proof}
Applying corollary of the well-known Allakhverdiyev theorem,  see   Corollary  2.2 \cite{firstab_lit:1Gohberg1965} (Chapter II, $\S$ 2.3), we get
$$
s_{(m+1)(k-1)+1}(B^{m+1})\leq s_{k}^{m+1}(B),\,m\in \mathbb{N}.
$$
Note that $n_{B}(r)=k,$ where $s_{k}^{-1}(B)<r<s_{k+1}^{-1}(B).$ Therefore
$$
s^{-1}_{(m+1)k+1}(B^{m+1})\geq s_{k+1}^{-m-1}(B)>r^{m+1},
$$
and we have
\begin{equation}\label{eq13a}
n_{B^{m+1}}(r^{m+1})\leq (m+1)k= (m+1) n_{B}(r),
\end{equation}
hence using \eqref{eq11}, we obtain
\begin{equation*}
   \frac{n_{B^{m+1}}(r^{m+1})}{r^{\alpha} }\rightarrow 0,\,r\rightarrow\infty,\, m=[\alpha].
\end{equation*}
The rest part of the proof is absolutely analogous to Theorem 2 \cite{firstab_lit:1kukushkin2021}.
 \end{proof}
\begin{teo}\label{T3} Assume that $B$ is a compact operator, $\Theta(B) \subset   \mathfrak{L}_{0}(\theta),\,\theta< \pi/2\alpha ,\;B\in \mathfrak{S}^{\star} _{\alpha},$
$$
 (\ln^{1+1/\alpha}x)'_{s^{-1}_{m}(B)}  =o(  m^{-1/\alpha}   ).
$$
where $\alpha$ is positive non integer.
Then the statement of the Theorem 4 \cite{firstab_lit:1kukushkin2021} remains true, i.e. statement S1 holds, moreover we have the following gaps between the eigenvalues
$$
|\lambda_{N_{\nu}+1}|^{-1}-|\lambda_{N_{\nu}}|^{-1}\geq K |\lambda_{N_{\nu}}|^{ \alpha-1},\,K>0,
$$
and the   eigenvalues corresponding to the partial sums become united to the groups  due to the distance of the power type
$$
|\lambda_{N_{\nu}+k}|^{-1}-|\lambda_{N_{\nu}+k-1}|^{-1}\leq K |\lambda_{N_{\nu}+k}|^{\alpha-1},\;
 k=2,3,...,N_{\nu+1}-N_{\nu}.
$$
\end{teo}
\begin{proof}
Using the theorem conditions, we obtain easily
$$
 s _{m}(B)   =o(  m^{-1/\alpha}   ),\;m\rightarrow\infty,
$$
  In accordance with   Corollary  3.2 \cite{firstab_lit:1Gohberg1965} (Chapter II, $\S$ 3.3),  it gives us
$$
|\lambda_{m}(B)|=o(  m^{-1/\alpha}   ),\;m\rightarrow\infty.
$$
In consequence of the latter relation, we have    $\mu_{m}/m^{ 1/\alpha}\geq C,$ where $\mu_{m}=\lambda^{-1}_{m}(B).$ Using the implication
$$
\lim\limits_{n\rightarrow\infty}(\mu_{n+1}-\mu_{n})/\mu^{(p-1)/p}_{n}=0,\;\Longrightarrow  \lim\limits_{n\rightarrow\infty} \mu_{n}/n^{p}=0,\;p>0,
$$
the proof of which can be found in the proof of Lemma 2 \cite{firstab_lit:2Agranovich1994}, we can claim that there exists    a subsequence $\{\lambda_{N_{\nu}}\}_{\nu=0}^{\infty},$
 such that
$$
|\mu_{N_{\nu}+1}|-|\mu_{N_{\nu}}|\geq K |\mu_{N_{\nu}}|^{1-\alpha},\,K>0.
$$
Indeed, for the case $p\geq1,$  making the auxiliary denotation $\xi_{n}:=\mu^{1/p}_{n},$ we can rewrite the assumed  implication   as follows
$$
\lim\limits_{n\rightarrow\infty}(\xi^{p}_{n+1}-\xi^{p}_{n})/\xi^{ p-1 }_{n}=0,\;\Longrightarrow  \lim\limits_{n\rightarrow\infty} \xi_{n}/n=0,
$$
denoting $\Delta\xi_{n}:=\xi _{n+1}-\xi _{n},$ we can easily obtain, due to the application of the mean value theorem (or Lagrange theorem), the estimate $\xi^{p}_{n+1}-\xi^{p}_{n}\geq p\, \xi^{p+1}_{n} \Delta\xi_{n}.$  It is clear that it gives us   the implication
$$
\left\{\lim\limits_{n\rightarrow\infty} \Delta\xi_{n}=0,\Rightarrow \lim\limits_{n\rightarrow\infty} \xi_{n}/n=0\right\}\Rightarrow\left\{\lim\limits_{n\rightarrow\infty}(\mu_{n+1}-\mu_{n})/\mu^{(p-1)/p}_{n}=0,\;\Longrightarrow  \lim\limits_{n\rightarrow\infty} \mu_{n}/n^{p}=0\right\},
$$
however  the fact that the implication in the first term holds   can be established with no difficulties.  The proof corresponding to the case $p<1$ is absolutely analogous, we just need notice
$$
\lim\limits_{n\rightarrow\infty}(\xi^{p}_{n+1}-\xi^{p}_{n})/\xi^{ p-1 }_{n}=0,\Rightarrow \lim\limits_{n\rightarrow\infty} \xi^{p}_{n} /\xi^{ p  }_{n+1}=1,
$$
and deal with the following implication
$$
\lim\limits_{n\rightarrow\infty}(\xi^{p}_{n+1}-\xi^{p}_{n})/\xi^{ p-1 }_{n+1}=0,\;\Longrightarrow  \lim\limits_{n\rightarrow\infty} \xi_{n}/n=0
$$
in the analogous way.

 Let us find $\delta_{\nu}$ from the condition $R_{\nu}=K|\mu_{N_{\nu}}|^{1-\alpha}+|\mu_{N_{\nu}}|,\,R_{\nu}(1-\delta_{\nu})=|\mu_{N_{\nu}}|,$ then $\delta_{\nu}^{-1}=1+K^{-1}|\mu_{N_{\nu}}|^{\alpha}.$ Note that by virtue of such a choice, we have $R_{\nu}<R_{\nu+1}(1-\delta_{\nu+1}).$

  In accordance with the theorem  condition,  established above relation \eqref{eq11a}, we have   $  n_{B}(r)=o( r^{\alpha}/\ln r).$
 Consider a contour
$$
 \vartheta(B):= \left\{\lambda:\;|\lambda|=r_{0}>0,\,|\mathrm{arg} \lambda|\leq \theta+\varepsilon\right\}\cup\left\{\lambda:\;|\lambda|>r_{0},\; |\mathrm{arg} \lambda|=\theta+\varepsilon\right\},
 $$
 here the number $r$ is chosen so that the  circle  with the corresponding radios does not contain values $\mu_{n},\,n\in \mathbb{N}.$
Applying  Lemma 5  \cite{firstab_lit:1kukushkin2021},  we have that there exists such a sequence $\{\tilde{R}_{\nu}\}_{0}^{\infty}:\;(1-\delta_{\nu }) R_{\nu}<\tilde{R}_{\nu}<R_{\nu }$ that the following estimate holds
$$
\|(I-\lambda B )^{-1}\|_{\mathfrak{H}}\leq e^{\gamma (|\lambda|)|\lambda|^{\alpha}}|\lambda|^{m},\,m=[\alpha],\,|\lambda|=\tilde{R}_{\nu},
$$
where
$$
\gamma(|\lambda|)= \beta ( |\lambda|^{m+1})  +(2+\ln\{4e/\delta_{\nu}\}) \beta(\sigma |\lambda|  ^{m+1}) \,\sigma ^{\alpha/(m+1)},\,\sigma:=\left(\frac{2e}{1-\delta_{0}}\right)^{m+1},\,|\lambda|=\tilde{R}_{\nu},
$$
$$
\;\beta(r )= r^{ -\frac{\alpha}{m+1} }\left(\int\limits_{0}^{r}\frac{n_{B^{m+1}}(t)dt}{t }+
r \int\limits_{r}^{\infty}\frac{n_{B^{m+1}}(t)dt}{t^{ 2  }}\right).
$$
Here we should explain that that the opportunity to obtain this estimate is based   upon  the estimation of the Fredholm  determinant $\Delta_{B^{m+1}}(\lambda^{m+1})$  absolute value  from bellow (see \cite[p.8]{firstab_lit:1Lidskii}). In its own turn the latter  is implemented via  the general estimate from bellow of the absolute value of a holomorphic function (Theorem 11 \cite[p.33]{firstab_lit:Eb. Levin}).
  However, to avoid any kind of misunderstanding let us implement a scheme of reasonings of  Lemma 5 consistently.\\

 Using  the theorem condition, we have  $B\in \mathfrak{S}_{\alpha+\varepsilon},\,\varepsilon>0.$
Obviously, we have
\begin{equation}\label{8a}
(I-\lambda B )^{-1}=(I-\lambda^{m+1}B^{m+1})^{-1}(I+\lambda B+\lambda^{2} B^{2}+...+\lambda^{m}B^{m}).
\end{equation}
In accordance with Lemma 3 \cite{firstab_lit:1Lidskii}, for sufficiently small $\varepsilon>0,$ we have
$$
\sum\limits_{n=1}^{\infty}s^{\frac{\alpha+\varepsilon}{m+1}}_{n}(  B^{m+1} )\leq \sum\limits_{n=1}^{\infty}s^{ \alpha+\varepsilon }_{n}( B )<\infty,
$$
       Applying  inequality  (1.27) \cite[p.10]{firstab_lit:1Lidskii} (since $  \rho  /(m+1)<1$), we get
$$
\|\Delta_{B^{m+1}}(\lambda^{m+1})(I-\lambda^{m+1} B^{m+1})^{-1}\| \leq C\prod\limits_{i=1}^{\infty}\{1+|\lambda^{m+1} s_{i}( B^{m+1} )|\},
$$
where $\Delta_{B^{m+1}}(\lambda^{m+1})$ is a Fredholm  determinant of the operator $B^{m+1}$  (see \cite[p.8]{firstab_lit:1Lidskii}).
Rewriting the formulas in accordance with the terms of the entire functions theory, we get
$$
  \prod\limits_{n=1}^{\infty}\{1+|\lambda^{m+1} s_{n}( B^{m+1} )|\}=\prod\limits_{n=1}^{\infty} G\left(  |\lambda|^{m+1}/ a_{n} ,p\right),
$$
where $a_{n}:=-s^{-1}_{n}(B^{m+1}),\,p=[\alpha /(m+1)]=0,$
using Lemma 2  \cite{firstab_lit:1kukushkin2021},   we get
$$
 \prod\limits_{n=1}^{\infty} G\left(  r^{m+1}/ a_{n} ,p\right)\leq  e^{ \beta (r^{m+1})r^{\alpha}},\,r>0.
$$
In accordance with Theorem 11 \cite[p.33]{firstab_lit:Eb. Levin},   and regarding to the case,   the following estimate holds
$$
|\Delta_{B^{m+1}}(\lambda^{m+1})|\geq e^{-(2+\ln\{2e/3\eta\})\ln\xi_{m}},\;\xi_{m}= \!\!\!\max\limits_{\psi\in[0,2\pi /(m+1)]}|\Delta_{B^{m+1}}([2e R_{\nu}e^{i\psi}]^{m+1})|,\,|\lambda|\leq R _{\nu},
$$
except for the exceptional set of circles with the sum of radii less that $4\eta R_{\nu},$  where $\eta$ is an arbitrary number less than  $2e/3 .$  Thus  to find the desired circle $\lambda= e^{i\psi}\tilde{R}_{\nu}$ belonging to the ring, i.e. $R_{\nu}(1-\delta_{\nu})<\tilde{R}_{\nu}<R_{\nu},$   we have to   choose $\eta$ satisfying the inequality
$$
4\eta R_{\nu}<R_{\nu}-R_{\nu}(1-\delta_{\nu})=\delta_{\nu} R_{\nu};\;\eta< \delta_{\nu}/4,
$$
for instance let us choose $\eta_{\nu}=\delta_{\nu}/6,$ here we should note that $\delta_{\nu}$ tends to zero, this is why without loss of generality of reasonings we are free in such a choice since the inequality $\eta_{\nu}<3 e/2$ would be satisfied for a sufficiently large $\nu.$ Under such assumptions, we can rewrite the above estimate in the form
$$
|\Delta_{B^{m+1}}(\lambda^{m+1})|\geq e^{-(2+\ln\{4e/\delta_{\nu}\})\ln\xi_{m}},\; |\lambda|=\tilde{R} _{\nu}.
$$
Note that in accordance with the estimate (1.21) \cite[p.10]{firstab_lit:1Lidskii}, Lemma 2  \cite{firstab_lit:1kukushkin2021} we have
$$
|\Delta_{B^{m+1}}(\lambda)|\leq  \prod\limits_{i=1}^{\infty}\{1+|\lambda s_{i}( B^{m+1} )|\}=\prod\limits_{n=1}^{\infty} G\left(  |\lambda|/ a_{n} ,p\right)\leq
 e^{ \beta (|\lambda| )|\lambda|^{\alpha/(m+1)}}.
$$
Using this estimate, we obtain
$$
\xi_{m}\leq e^{ \beta ([2e  R_{\nu}]^{m+1})( 2e R_{\nu}  )^{\alpha}}.
$$
Substituting, we get
$$
|\Delta_{B^{m+1}}(\lambda^{m+1})|^{-1}\leq e^{(2+\ln\{4e/\delta_{\nu}\})\beta ([2e  R_{\nu}]^{m+1})( 2e R_{\nu}  )^{\rho}},\; |\lambda|=\tilde{R} _{\nu}.
$$
Having noticed the facts
$$
 1-\delta_{\nu}  = \frac{|\mu_{N_{\nu}}|^{\alpha}}{K+|\mu_{N_{\nu}}|^{\alpha}}=1-\frac{K}{K+|\mu_{N_{\nu}}|^{\alpha}}\geq 1-\frac{K}{K+|\mu_{N_{0}}|^{\alpha}}= 1-\delta_{0} ;
 $$
 $$
 \;R_{\nu}<\tilde{R}_{\nu}(1-\delta_{\nu})^{-1}\leq \tilde{R}_{\nu}(1-\delta_{0})^{-1},
$$
we obtain
$$
|\Delta_{B^{m+1}}(\lambda^{m+1})|^{-1}\leq e^{(2+\ln\{4e/\delta_{\nu}\})\beta (\sigma \tilde{R}^{m+1}_{\nu} )   \tilde{R}^{\alpha}_{\nu} \sigma^{\alpha/(m+1)}  },\; |\lambda|=\tilde{R} _{\nu}.
$$
Combining the above estimates, we get

$$
 \|\Delta_{B^{m+1}}(\lambda^{m+1})(I-\lambda^{m+1} B^{m+1})^{-1}\| =|\Delta_{B^{m+1}}(\lambda^{m+1})|\cdot\|(I-\lambda^{m+1} B^{m+1})^{-1}\| \leq e^{ \beta (|\lambda|^{m+1})|\lambda|^{\alpha}};
$$
$$
\|(I-\lambda^{m+1} B^{m+1})^{-1}\|\leq e^{ \beta (|\lambda|^{m+1})|\lambda|^{\alpha}}|\Delta_{B^{m+1}}(\lambda^{m+1})|^{-1}=
e^{\gamma (|\lambda|)|\lambda|^{\alpha}}  ,\,|\lambda|=\tilde{R}_{\nu}.
$$
Consider relation \eqref{8a}, we have
$$
 \|(I-\lambda B )^{-1}\| \leq\|(I-\lambda^{m+1}B^{m+1})^{-1}\|  \cdot\|(I+\lambda B+\lambda^{2} B^{2}+...+\lambda^{m}B^{m})\| \leq
$$
$$
\leq\|(I-\lambda^{m+1}B^{m+1})^{-1}\|  \cdot \frac{|\lambda|^{m+1}\|B\|^{m+1}-1}{|\lambda|\cdot\|B\|-1}\leq C e^{\gamma (|\lambda|)|\lambda|^{\alpha}}|\lambda|^{m},\,|\lambda|=\tilde{R}_{\nu}.
$$
Applying the obtained estimate for the norm,  we can  claim that  for a sufficiently   small $\delta>0,$   there exists an arch
$\tilde{\gamma}_{ \nu }:=\{\lambda:\;|\lambda|=\tilde{R}_{\nu},\,|\mathrm{arg } \lambda|< \theta +\varepsilon\}$    in the ring $(1-\delta_{\nu})R_{\nu}<|\lambda|<R_{\nu},$ on which the following estimate holds
$$
 J_{  \nu  }: =\left\|\,\int\limits_{\tilde{\gamma}_{ \nu }}e^{-\lambda^{\alpha}t}B(I-\lambda B)^{-1}f d \lambda\,\right\| \leq \,\int\limits_{\tilde{\gamma}_{ \nu }}e^{- t \mathrm{Re}\lambda^{\alpha}}\left\|B(I-\lambda B)^{-1}f \right\|  |d \lambda|\leq
$$
 $$
 \leq \|f\| C  e^{\gamma (|\lambda|)|\lambda|^{\alpha} }|\lambda|^{m }\int\limits_{-\theta-\varepsilon}^{\theta+\varepsilon} e^{- t \mathrm{Re}\lambda^{\alpha}} d \,\mathrm{arg} \lambda,\,|\lambda|=\tilde{R}_{\nu}.
$$
Using the theorem conditions,    we get     $\,|\mathrm{arg} \lambda |<\pi/2\alpha,\,\lambda\in \tilde{\gamma}_{ \nu },\,\nu=0,1,2,...\,,$ we get
$$
\mathrm{Re }\lambda^{\alpha}\geq |\lambda|^{\alpha} \cos \left[(\pi/2\alpha-\delta)\alpha\right]= |\lambda|^{\alpha} \sin     \alpha \delta,
$$
 where $\delta$ is a sufficiently small number.  Therefore
\begin{equation}\label{13b}
 J_{\nu} \leq C e^{|\lambda|^{\alpha}\{\gamma (|\lambda|)-t  \sin     \alpha \delta\}}|\lambda|^{m }
 ,\,m=[\alpha],\,|\lambda|=\tilde{R}_{\nu}.
\end{equation}
It is clear that within the contour $\vartheta(B),$ between the arches $\tilde{\gamma}_{ \nu },\tilde{\gamma}_{ \nu+1 }$ (we denote the boundary of this  domain by $\gamma_{ \nu}$) there  lie the eigenvalues only for which the following relation holds
$$
|\mu_{N_{\nu}+k }|-|\mu_{N_{\nu}+k-1 }|\leq K |\mu_{N_{\nu}+k-1 }|^{1-\alpha},\;
 k=2,3,...,N_{\nu+1}-N_{\nu} .
$$
Using Lemma 8, \cite{firstab_lit:1kukushkin2021},  we   obtain a relation
$$
 \frac{1}{2\pi i}\int\limits_{\gamma_{ \nu}}e^{-\lambda^{\alpha}t}B(I-\lambda B)^{-1}f d \lambda
=     \sum\limits_{q=N_{\nu}+1}^{N_{\nu+1}}\sum\limits_{\xi=1}^{m(q)}\sum\limits_{i=0}^{k(q_{\xi})}e_{q_{\xi}+i}c_{q_{\xi}+i}(t).
$$
It is clear that to obtain the desired result, we should prove that the series composed  of the above terms converges. Consider the following auxiliary denotations originated from the reasonings $\gamma_{\nu_{+}}:=
\{\lambda:\, \tilde{R}_{\nu}<|\lambda|<\tilde{R}_{\nu+1},\, \mathrm{arg} \lambda  =\theta   +\varepsilon\},\,\gamma_{\nu_{-}}:=
\{\lambda:\,\tilde{R}_{\nu}<|\lambda|<\tilde{R}_{\nu+1},\, \mathrm{arg} \lambda  =-\theta   -\varepsilon\},$
$$
J^{+}_{\nu}: =\left\|\,\int\limits_{\gamma_{\nu_{+}}}e^{-\lambda^{\alpha}t}B(I-\lambda B)^{-1}f d \lambda\,\right\|,\; J^{-}_{\nu}:= \left\|\,\int\limits_{\gamma_{\nu_{-}}}e^{-\lambda^{\alpha}t}B(I-\lambda B)^{-1}f d \lambda\,\right\|,
$$
wee have
$$
\left\| \int\limits_{\gamma_{ \nu}}e^{-\lambda^{\alpha}t}B(I-\lambda B)^{-1}f d \lambda\right\|\leq J_{\nu}+J_{\nu+1}+J^{+}_{\nu}+J^{-}_{\nu}.
 $$
Thus, it is clear that to establish S1, it suffices to establish the facts
\begin{equation}\label{eq16}
\sum\limits_{\nu=0}^{\infty}J_{\nu}<\infty,\;\sum\limits_{\nu=0}^{\infty}J^{+}_{\nu}<\infty,\;\sum\limits_{\nu=0}^{\infty}J^{-}_{\nu}<\infty.
\end{equation}
Consider the right-hand side of formula \eqref{13b}.
Substituting $\delta^{-1}_{\nu},$ we have
$$
\ln\{4e/\delta_{\nu}\} = 1+ \ln4 +\ln \{1+K^{-1}|\mu_{N_{\nu}}|^{\alpha}\} \leq C \ln   |\mu_{N_{\nu}}|     .
$$
Hence, to obtain the desired result we should prove that $ \beta (\sigma| \mu_{N_{\nu}}| ^{m+1})\ln  |\mu_{N_{\nu}}|\rightarrow 0,\,\nu\rightarrow\infty.$ Note that in accordance with relation \eqref{eq7}, we can prove the latter relation, if we show that
\begin{equation}\label{13a}
\ln r\frac{n_{B^{m+1}}(r^{m+1})}{r^{\alpha} }\rightarrow 0,\;r\rightarrow\infty.
\end{equation}
In accordance with \eqref{eq11a}, we have
\begin{equation*}
(\ln^{1+1/\alpha}x)'_{s^{-1}_{m}(B)}  =o(  m^{-1/\alpha}   )
 , \Longrightarrow \left(\ln r\frac{n(r)}{r^{\rho}}\,\rightarrow 0,\,r\rightarrow\infty\right),
\end{equation*}
Applying \eqref{eq13a}, we have
$$
n_{B^{m+1}}(r^{m+1})\leq (m+1) n_{B}(r),
$$
hence
$$
\ln r\frac{n_{B^{m+1}}(r^{m+1})}{r^{\alpha} }\leq (m+1)\ln r\frac{n_{B }(r )}{r^{\alpha} },
$$
what gives us the desired result, i.e. we compleat the proof of the first relation \eqref{eq16}.

In accordance with Lemma 6 \cite{firstab_lit:1kukushkin2021}, we can claim  that   on each ray $\zeta$ containing the point zero and not belonging to the sector $\mathfrak{L}_{0}(\theta)$ as well as the  real axis, we have
$$
\|(I-\lambda B)^{-1}\|\leq \frac{1}{\sin\psi},\,\lambda\in \zeta,
$$
where $\,\psi = \min \{|\mathrm{arg}\zeta -\theta|,|\mathrm{arg}\zeta +\theta|\}.$ Applying this estimate, the established above estimate $\mathrm{Re }\lambda^{\alpha}\geq |\lambda|^{\alpha} \sin     \alpha \delta,$ we get
$$
 J^{+}_{\nu} \leq C\|f\|  \!\!\! \int\limits_{ \tilde{R}_{\nu} }^{\tilde{R}_{\nu+1} }  e^{- t \mathrm{Re }\lambda^{\alpha}}   |d   \lambda|\leq C \|f\| e^{-t \tilde{R}_{\nu}  ^{\alpha} \sin     \alpha\delta}   \!\!\!\int\limits_{ \tilde{R}_{\nu} }^{\tilde{R}_{\nu+1} }   |d   \lambda|=
 $$
 $$
 = C\|f\| e^{-t \tilde{R}_{\nu}  ^{\alpha} \sin     \alpha\delta}( \tilde{R}_{\nu+1}-\tilde{R}_{\nu}) .
$$
It obvious, that the same estimate can be obtained for $J^{-}_{\nu}.$ Therefore, the second and the third relations \eqref{eq16} hold. The proof is complete.
 \end{proof}
\begin{remark} Application of the results established in paragraph 1,  under the imposed sectorial condition upon   the compact  operator $B,$   gives us
$$
\left\{(\ln^{1+1/\rho}x)'_{\lambda^{-1}_{m}(\mathfrak{Re}B)}  =o(  m^{-1/\rho}   )\right\}\Longrightarrow\left\{(\ln^{1+1/\rho}x)'_{s^{-1}_{m}(B)}  =o(  m^{-1/\rho}   )\right\},
$$
what becomes clear if we recall $s_{m}(B)\leq C \lambda_{m}(\mathfrak{Re}B).$ However to establish the equivalence
$$
B\in\mathfrak{S}^{\star}_{\rho}\Longleftrightarrow \mathfrak{Re} B\in\mathfrak{S}^{\star}_{\rho},
$$  we require the estimate from bellow $s_{m}(B)\geq C \lambda_{m}(\mathfrak{Re}B),$ which in its own turn can be established due to a more subtle technique and under  stronger  conditions regarding the operator $B.$ In particular the second representation theorem is involved and conditions upon the matrix coefficients of the operator $B$ are imposed in accordance with which they should decrease sufficiently fast. In addition,  the application of Theorem 5 \cite{firstab_lit(arXiv non-self)kukushkin2018}  (if $B$ satisfies its conditions)  gives us the relation
$
\lambda_{n}(\mathfrak{Re}   B)\asymp \lambda_{n}(H^{-1}),\,H=\mathfrak{Re}B^{-1},
$
the latter allows one to deal with unbounded operators reformulating Theorem \ref{T3} in terms of the  Hermitian real component. It should be noted that the made remark  remains true regarding the implication
$$
\left\{\lambda_{m}(\mathfrak{Re} B)=o(m^{-1/\rho})\right\}\Longrightarrow  \left\{s_{m}(B)=o(m^{-1/\rho})\right\},
$$
and since the scheme of the reasonings is the same  we left  them  to the reader.
\end{remark}

\begin{ex}\label{E2} Using decomposition \eqref{eq6},  let us construct artificially  the sectorial operator $B$ satisfying the Theorem \ref{T2} condition. Observe the produced above example $s_{m}(B)=(m \ln m)^{-1/\rho},$ as it was said above the latter condition guarantees $B\in \mathfrak{S}_{\rho}^{\star}.$ However, the problem is how to choose in the polar decomposition formula  the unitary operator  $U$ providing the sectorial property of the operator $B.$  The following approach is rather rough but  can be used to supply the desired example, in terms of  formula  \eqref{eq6}, we get
$$
 (Bf,f)=\sum\limits_{n=1}^{\infty}s_{n}(B)(f,e_{n})\overline{(f,g_{n})}.
 $$
 Therefore, if we impose the condition
 $$
 \left|\arg (f,e_{n})-   \arg (f,g_{n})\right|<\theta,\,n\in \mathbb{N},\,f\in \mathrm{R}(B),
 $$
 then we obtain the desired sectorial condition, where $\theta$ is a semi-angle of the sector. It is remarkable that that the selfadjoint case corresponds to the zero difference since in this case $U=I.$ However, we should remark that following to the classical mathematical style it is necessary to formulate conditions of the sectorial property in terms of the eigenvectors of the operator $|B|$ but it is not so easy matter requiring comprehensive analysis  of the operator $U$ properties provided with   concrete cases.
\end{ex}


\begin{thebibliography}{2}







\bibitem{firstab_lit:2Agranovich1994} {\sc  Agranovich M.S.} On series with respect to root vectors of operators associated with forms having symmetric principal part.   \textit{ Functional Analysis and its applications},   \textbf{28} (1994),   151-167.







\bibitem{firstab_lit: Ahiezer1966}  {\sc Ahiezer N.I., Glazman I.M.}  Theory of linear operators in a Hilbert space.  \textit{Moscow: Nauka, Fizmatlit}, 1966.












 \bibitem{firstab_lit:Bazhl} {\sc Bazhlekova  E.} The abstract Cauchy problem for the fractional evolution
equation.
\textit{Fractional Calculus and Applied Analysis},   \textbf{1}, No.3 (1998),  255-270.




\bibitem{Ph. Cl} {\sc Cl\'{e}ment  Ph., Gripenberg  G.,  Londen  S-O.} H\"{o}lder regularity
for a linear fractional evolution equation.
\textit{Topics in Nonlinear Analysis, The Herbert
Amann Anniversary Volume}, Birkh\"{a}user, Basel (1998).




\bibitem{firstab_lit:2Dim-Kir} {\sc Dimovski I.H., Kiryakova V.S.} Transmutations, convolutions and fractional powers of Bessel-type operators via Maijer's  G-function. In:
\textit{Proc."Complex Anall. and Appl-s, Varna' 1983"}, (1985),  45-66.


\bibitem{firstab_lit:15Erdelyi} {\sc Erdelyi A.} Fractional integrals of generalized functions.
\textit{J. Austral. Math. Soc.},  \textbf{14}, No.1  (1972),  30-37.




\bibitem{firstab_lit:Goldstein} {\sc   Goldstein G.R.,  Goldstein J.A.,  Rhandi A.}  Weighted Hardys inequality
and the Kolmogorov equation perturbed by an inverse-square potential.  \textit{Applicable Analysis}, \textbf{91}, Issue: 11 (2012), DOI:10.1080/00036811.2011.587809.



\bibitem{firstab_lit:1Gohberg1965} {\sc   Gohberg I.C., Krein M.G.}  Introduction to the theory of linear non-selfadjoint operators in a Hilbert space.
 \textit{Moscow: Nauka, Fizmatlit},  1965.


\bibitem{firstab_lit:Hardy}
 {\sc Hardy G.H.} Divergent series.
  \textit{Oxford University Press, Ely House, London W.}, 1949.




\bibitem{firstab_lit:kato1980}{\sc Kato T.} Perturbation theory for linear operators. \textit{Springer-Verlag Berlin, Heidelberg, New York}, 1980.

\bibitem{firstab_lit:1Katsnelson} {\sc Katsnelson V.E.} Conditions under which systems of eigenvectors of some classes
of operators form a basis.
\textit{ Funct. Anal. Appl.}, \textbf{1}, No.2   (1967), 122-132.


 \bibitem{firstab_lit:kipriyanov1960}
{\sc Kipriyanov I.A. } On spaces of fractionally differentiable functions. \textit{ Proceedings of the Academy of Sciences. USSR},  \textbf{24} (1960),  665-882.

\bibitem{firstab_lit:1kipriyanov1960}
{\sc Kipriyanov I.A. }  The operator of fractional differentiation and powers of the elliptic operators. \textit{  Proceedings of the Academy of Sciences. USSR},    \textbf{131} (1960),     238-241.





















 \bibitem{firstab_lit:1kukushkin2018}   {\sc  Kukushkin M.V.} Spectral properties of  fractional differentiation operators. \textit{Electronic Journal of Differential Equations}, \textbf{ 2018}, No. 29 (2018),   1-24.



\bibitem{firstab_lit(arXiv non-self)kukushkin2018}   {\sc  Kukushkin M.V.} On One Method of Studying Spectral Properties of Non-selfadjoint Operators. \textit{Abstract and Applied Analysis; Hindawi: London, UK} \textbf{2020},  (2020); at  https://doi.org/10.1155/2020/1461647.


\bibitem{kukushkin2019}{\sc Kukushkin M.V.} Asymptotics of eigenvalues for differential
operators of fractional order.\textit{Fract. Calc. Appl. Anal.} \textbf{22}, No. 3 (2019), 658--681, arXiv:1804.10840v2 [math.FA]; DOI:10.1515/fca-2019-0037; at https://www.degruyter.com/view/j/fca.


\bibitem{kukushkin2021a}{\sc Kukushkin M.V.} Abstract fractional calculus for m-accretive operators.
 \textit{International Journal of Applied Mathematics.} \textbf{34}, Issue: 1 (2021),  DOI: 10.12732/ijam.v34i1.1

\bibitem{firstab_lit:1kukushkin2021}   {\sc  Kukushkin M.V.} Natural lacunae method and Schatten-von Neumann classes of the convergence exponent. \textit{Mathematics}  (2022), \textbf{10}, (13), 2237; https://doi.org/10.3390/math10132237.

\bibitem{firstab_lit:2kukushkin2022}   {\sc Kukushkin M.V.} Evolution Equations in Hilbert Spaces via the Lacunae Method. \textit{Fractal Fract.}  (2022), \textbf{6}, (5), 229; https://doi.org/10.3390/fractalfract6050229.

\bibitem{firstab_lit(axi2022)}   {\sc Kukushkin M.V.} Abstract Evolution Equations with an Operator Function in the Second Term. \textit{Axioms} (2022), 11, 434; at  https://doi.org/10.3390/axioms
11090434.

\bibitem{firstab_lit KRAUNC}   {\sc Kukushkin M.V.} Note on the spectral theorem for unbounded non-selfadjoint
operators. \textit{Vestnik KRAUNC. Fiz.-mat. nauki.} (2022), 139, 2, 44-63. DOI: 10.26117/2079-
6641-2022-39-2-44-63




\bibitem{firstab_lit(frac2023)}{\sc M.V. Kukushkin}  Cauchy Problem for an Abstract Evolution Equation of Fractional Order. \textit{Fractal Fract.} (2023), 7, 111; at   https://doi.org/10.3390/fractalfract7020111.


\bibitem{firstab_lit:Krasnoselskii M.A.}{\sc Krasnoselskii M.A.,  Zabreiko P.P.,  Pustylnik E.I.,  Sobolevskii P.E.}
  Integral operators in the spaces of summable functions. \textit{ Moscow:  Science,   FIZMATLIT},   1966.


\bibitem{firstab_lit:1Krein} {\sc  Krein M.G.}  Criteria for
completeness of the system of root vectors of a dissipative operator.
 \textit{Amer. Math. Soc. Transl. Ser., Amer. Math. Soc., Providence, RI}, \textbf{26}, No.2 (1963), 221-229.


\bibitem{firstab_lit:Fan} {\sc  Ky Fan}  Maximum properties and inequalities for the eigenvalues of completely continuous operators.
 \textit{Proc. Nat. Acad. Sci. USA}, \textbf{37}, (1951), 760--766.




\bibitem{firstab_lit:Eb. Levin} {\sc  Levin  B. Ja.}  Distribution of Zeros of Entire Functions.
 \textit{Translations of Mathematical Monographs},     1964.


\bibitem{firstab_lit:1Lidskii} {\sc  Lidskii V.B.} Summability of series in terms of the principal vectors of non-selfadjoint operators.
 \textit{Tr. Mosk. Mat. Obs.}, \textbf{11},  (1962), 3-35.















\bibitem{firstab_lit:Markus Matsaev} {\sc Markus A.S.,  Matsaev V.I.} Operators generated by sesquilinear forms and their spectral asymptotics.
\textit{ Linear operators and integral equations, Mat. Issled., Stiintsa, Kishinev}, \textbf{61}   (1981), 86-103.





\bibitem{firstab_lit:2Markus} {\sc Markus A.S.} Expansion in root vectors of a slightly perturbed selfadjoint operator.
\textit{ Soviet Math. Dokl.}, \textbf{3}   (1962), 104-108.




\bibitem{firstab_lit:Mainardi F.} {\sc Mainardi F. } The fundamental solutions for the fractional diffusion-wave equation.
\textit{Appl. Math. Lett.}, \textbf{9}, No.6  (1966),  23-28.


\bibitem{firstab_lit:Mamchuev2017} {\sc Mamchuev M.O. } Solutions of the main boundary value problems for the time-fractional telegraph equation by the Green function method.
\textit{Fractional Calculus and Applied Analysis}, \textbf{20}, No.1  (2017),  190-211,   DOI: 10.1515/fca-2017-0010.




\bibitem{firstab_lit:Mamchuev2017a} {\sc Mamchuev M.O. } Boundary value problem for the time-fractional telegraph equation with Caputo derivatives Mathematical Modelling of Natural Phenomena.
\textit{Special functions and analysis of PDEs}, \textbf{12}, No.3  (2017),  82-94, DOI: 10.1051/mmnp/201712303.


\bibitem{firstab_lit:9McBride} {\sc McBride A.} A note of the index laws of fractional calculus.
\textit{J. Austral. Math. Soc.}, A \textbf{34}, No.3  (1983),  356-363.



\bibitem{L. Mor} {\sc  Moroz L., Maslovskaya A. G.} Hybrid stochastic fractal-based approach to modeling the switching kinetics of ferroelectrics in the injection mode.
\textit{ Mathematical Models and Computer Simulations}, \textbf{12}   (2020), 348-356.










\bibitem{firstab_lit:mihlin1970}{\sc Mihlin S.G.} Variational methods in mathematical physics. \textit{Moscow Science}, 1970.









\bibitem{firstab_lit:1Nakhushev1977} {\sc Nakhushev A.M.}  The Sturm-Liouville problem for an ordinary differential equation of the  second order with  fractional derivatives in  lower terms.
 \textit{Proceedings of the Academy of Sciences. USSR},  \textbf{234}, No.2 (1977),  308-311.

\bibitem{firstab_lit:nakh2003}
 {\sc Nakhushev A.M.} Fractional calculus and its application.
  \textit{M.: Fizmatlit}, 2003.

\bibitem{firstab_lit:Pskhu} {Pskhu  A.V.} The fundamental solution of a diffusion-wave equation of fractional order.
\textit{Izvestiya: Mathematics},  \textbf{73}, No.2  (2009),  351-392.



\bibitem{firstab_lit:Rosenblum} {\sc  Rozenblyum G.V.,  Solomyak M.Z., Shubin M.A.
} Spectral theory of differential operators.
 \textit{Results of science and technology. Series Modern problems of mathematics
Fundamental directions}, \textbf{64}    (1989),    5-242.








\bibitem{firstab_lit:Riesz1955} {\sc  Riesz F.,  Nagy B. Sz.} Functional Analysis.
  \textit{Ungar, New York}, 1955.

\bibitem{firstab_lit:samko1987} {\sc Samko S.G., Kilbas A.A., Marichev O.I.} Fractional Integrals and Derivatives: Theory and Applications.
  \textit{Gordon
and Breach Science Publishers: Philadelphia, PA, USA}, 1993.

\bibitem{firstab_lit:Shkalikov A.} {\sc  Shkalikov A.A.}  Perturbations of selfadjoint and normal operators with a discrete spectrum.
 \textit{Russian Mathematical Surveys}, \textbf{71}, Issue 5(431) (2016),  113-174.

\bibitem{firstab_lit:Wyss} {\sc  Wyss W.}  The fractional diffusion equation.
 \textit{J. Math. Phys.}, \textbf{27},  No.11  (1986),  2782-2785.






\end{thebibliography}
\end{document}